\documentclass[11pt]{article}
\usepackage[letterpaper,twoside,outer=1.3in,vmargin=1.3in,]{geometry}
\usepackage{amsmath,amssymb,amsthm,amsfonts,enumerate,times,epsfig,color,hyperref,cleveref,subcaption}
\usepackage{thmtools}
\usepackage{thm-restate}
\usepackage{tikz}
\usetikzlibrary{positioning}
\usepackage{soul}
\allowdisplaybreaks

\newcommand{\ignore}[1]{}

\DeclareMathOperator{\cone}{cone}
\DeclareMathOperator{\lin}{lin}
\DeclareMathOperator{\lat}{lat}
\DeclareMathOperator{\scr}{SCR}
\DeclareMathOperator{\dij}{DIJ}
\DeclareMathOperator{\aff}{aff}
\DeclareMathOperator{\supp}{supp}

\DeclareMathOperator{\ind}{ind}

\DeclareMathOperator{\disc}{disc}
\DeclareMathOperator{\sources}{sources}
\DeclareMathOperator{\sinks}{sinks}

\newcommand{\1}{\mathbf{1}}
\newcommand{\0}{\mathbf{0}}

\newcommand{\cR}{\mathbb R}
\newcommand{\cZ}{\mathbb Z}

\newcommand{\zU}{\mathcal U}

\newcommand{\zC}{\mathcal C}
\newcommand{\zF}{\mathcal F}

\newcommand{\zE}{\mathcal E}

\newtheorem{theorem}{Theorem}[section]
\newtheorem{CO}[theorem]{Corollary}
\newtheorem{LE}[theorem]{Lemma}

\newtheorem{CN}[theorem]{Conjecture}
\newtheorem{RE}[theorem]{Remark}

\newtheorem{DE}[theorem]{Definition}

\newcounter{claim_nb}[theorem]
\newtheorem{claim}[claim_nb]{Claim}
\newtheorem*{claim*}{Claim}
\newtheorem*{subclaim*}{Subclaim}

\newcounter{claim_nbs}[section]
\newcounter{subclaim_nb}[claim_nbs]
\newenvironment{cproof}
{\begin{proof}
 [Proof of Claim.]
 \vspace{-1.2\parsep}}
{\renewcommand{\qed}{\hfill $\Diamond$} \end{proof}}

\newif\ifnotes\notesfalse 

\ifnotes
\usepackage{color}
\newcommand{\notename}[2]{{\textcolor{red}{\footnotesize{\bf (#1:} {#2}{\bf ) }}}}

\newcommand{\anote}[1]{{\notename{Ahmad}{#1}}}
\newcommand{\gnote}[1]{{\notename{Gerard}{#1}}}
\newcommand{\onote}[1]{{\notename{Olha}{#1}}}
\newcommand{\snote}[1]{{\notename{Siyue}{#1}}}
\renewcommand{\b}[1]{{\color{blue} #1}}
\renewcommand{\r}[1]{{\color{red} #1}}

\else

\newcommand{\notename}[2]{{}}

\newcommand{\gnote}[1]{}
\newcommand{\anote}[1]{}
\newcommand{\onote}[1]{}
\newcommand{\snote}[1]{}
\renewcommand{\b}[1]{#1}
\renewcommand{\r}[1]{}

\fi

\title{Strongly connected orientations and integer lattices}

\author{Ahmad Abdi \and G\'{e}rard Cornu\'{e}jols \and Siyue Liu \and Olha Silina}
\begin{document}

\maketitle

\begin{abstract}
Let $D=(V,A)$ be a digraph whose underlying \b{undirected} graph is $2$-edge-connected, and let $P$ be the polytope whose vertices are the incidence vectors of arc sets whose reversal makes $D$ strongly connected. We study the lattice theoretic properties of the integer points contained in a proper face $F$ of $P$ not contained in $\{x:x_a=i\}$ for any $a\in A,i\in \{0,1\}$. We prove under a mild necessary condition that $F\cap \{0,1\}^A$ contains an \emph{integral basis} $B$, i.e., $B$ is linearly independent, and any integral vector in the linear hull of $F$ is an integral linear combination of $B$. This result is surprising as the integer points in $F$ do not necessarily form a Hilbert basis. In proving the result, we develop a theory similar to Matching Theory for degree-constrained dijoins in bipartite digraphs. Our result has consequences for head-disjoint strong orientations in hypergraphs, and also to a famous conjecture by Woodall that the minimum size of a dicut of $D$, say~$\tau$, is equal to the maximum number of disjoint dijoins. We prove a relaxation of this conjecture, by finding for any prime number $p\geq 2$, a $p$-adic packing of dijoins of value $\tau$ and of support size at most $2|A|$. We also prove that the all-ones vector belongs to the lattice generated by $F\cap \{0,1\}^A$, where $F$ is the face of $P$ satisfying $x(\delta^+(U))=1$ for every dicut $\delta^+(U)$ \b{with minimum size}.\\

\end{abstract}


\section{Introduction}\label{sec:intro}

Let $D=(V,A)$ be a digraph whose underlying undirected graph is $2$-edge-connected. A \emph{strengthening set} is an arc subset $J$ such that the digraph obtained from $D$ after reversing the arcs in $J$ is strongly connected. Observe that $J\subseteq A$ is a strengthening set if, and only if, its indicator vector $x$ satisfies the following \emph{generalized set covering inequalities}: \begin{equation}\label{cut-ineq}\tag{CUT}
	\sum_{a\in \delta^+(U)}x_a + \sum_{b\in \delta^-(U)}(1-x_b)\geq 1 \quad \forall U\subset V,U\neq \emptyset.
\end{equation} In words, \eqref{cut-ineq} asks that after reversing the arcs of $J$ in $D$, every nonempty proper node subset $U$ has at least one incoming arc. Observe that \eqref{cut-ineq} can be rewritten as $x(\delta^+(U))-x(\delta^-(U))\geq 1-|\delta^-(U)|$; as the right-hand sides correspond to a crossing supermodular function, the system above may be viewed as a \emph{supermodular flow system}. Let $$
\scr(D):=[0,1]^A \cap 
\left\{
x: x \text{ satisfies \eqref{cut-ineq}}
\right\}\b{,}
$$ 
\b{where $\scr$ is an acronym for strongly connected re-orientations.}
It is well-known that $\scr(D)$ is a nonempty integral polytope, and so its vertices are precisely the indicator vectors of the strengthening sets of $D$ (\cite{Edmonds77}, see~\cite{Schrijver03}, \S60.1). This polytope and its variants have played an important role in graph orientations, combinatorial and matroid optimization; see (\cite{Schrijver03}, Chapters 60-61) and (\cite{Frank11}, Chapter 16).

In this paper, we study the lattice theoretic properties of the integer points in $\scr(D)$. Given a rational linear subspace $S\subseteq \cR^A$, an \emph{integral basis for $S$} is a subset $B\subseteq S\cap \cZ^A$ of linearly independent vectors such that every vector in $S\cap \cZ^A$ is an integral linear combination of $B$.

\begin{theorem}\label{scr-theorem}
	Let $D=(V,A)$ be a digraph whose underlying undirected graph is $2$-edge-connected. Let $\zF$ be a nonempty family over ground set $V$ such that $\emptyset,V\notin \zF$, and the following face of $\scr(D)$ is nonempty:
	$$F:=\scr(D)\cap \left\{x\in \cR^A: x(\delta^+(U))-x(\delta^-(U))=1-|\delta^-(U)|,\, \forall U\in \zF\right\}.$$
	Suppose $\gcd\{1-|\delta^-(U)|:U\in \zF\}=1$. Then $F\cap \{0,1\}^A$ contains an integral basis for $\lin(F)$. 
	\end{theorem}
	
Above, $\lin(F)$ refers to the linear hull of $F$. It can be readily checked that the GCD condition is necessary for $F\cap \{0,1\}^A$ to contain an integral basis for $\lin(F)$. \Cref{scr-theorem} is a consequence of a more general theorem about the lattice generated by the integer points in any face of $\scr(D)$ where the GCD condition is replaced by `$1-|\delta^-(U)|\neq 0$ for some $U\in \zF$'. This theorem is stated in \S\ref{subsec:scr-theorem-proof}.

\Cref{scr-theorem} is best possible in two different ways. First, the result does not extend to faces $F$ involving both $\0\leq x\leq \1$ and \eqref{cut-ineq} inequalities. Secondly, for the face $F$ from \Cref{scr-theorem}, the integer points in $F$ do not necessarily form a \emph{Hilbert basis}, so the result cannot be strengthened in this direction either. We shall explain both of these points further in \S\ref{subsec:schrijver-example}, where we also conjecture an extension of \Cref{scr-theorem} to faces of $\scr(D)$ where some capacity constraints are also fixed to equality.

\subsection{Three applications of the main theorem}

Let us discuss some applications of our result. 

\paragraph{Woodall's conjecture.} Let $D=(V,A)$ be a digraph whose underlying undirected graph is connected. A \emph{dicut} is the set of arcs leaving a nonempty proper node subset with no incoming arc, i.e., it is of the form $\delta^+(U)$ where $U\subset V, U\neq \emptyset$ and $\delta^-(U)=\emptyset$. A \emph{dijoin} is an arc subset whose contraction makes the digraph strongly connected. Subsequently, every strengthening set is a dijoin. It can be readily checked that $J$ is a dijoin if, and only if, $J$ intersects every dicut at least once \b{if, and only if, the digraph obtained from $D$ after bidirecting the arcs in $J$ is strongly connected. See \Cref{fig:dijoin-dicut} for an illustration of these concepts.}

A famous conjecture by Douglas Woodall states that the maximum number of \b{(arc)} disjoint dijoins is equal to the minimum size \b{(i.e., number of arcs)} of a dicut~\cite{Woodall78}. This conjecture has a convenient reformulation that appears in an unpublished note by Lex Schrijver.

\begin{CN}[\cite{Schrijver-note}]\label{Woodall-CN-2}
	Let $\tau\geq 2$ be an integer, and let $D=(V,A)$ be a digraph, where every dicut has size at least $\tau$. Then $A$ can be partitioned into $\tau$ strengthening sets.
\end{CN}

Note the difference between the original formulation of Woodall's conjecture and \Cref{Woodall-CN-2}. While the former is concerned with \emph{packing} dijoins, the latter seeks a \emph{partition} into strengthening sets. This subtle difference comes from the key distinction that while every superset of a dijoin is also a dijoin, a superset of a strengthening set may not remain a strengthening set. 

As a consequence of \Cref{scr-theorem}, we obtain the following relaxation of this conjecture. For a subset $J\subseteq A$, denote by $\1_J\in \{0,1\}^A$ the indicator vector of $J$. \b{A \emph{minimum dicut} is a dicut with minimum size.}

\begin{theorem}\label{ARF-partition-CO}
	Let $\tau\geq 2$ be an integer, and let $D=(V,A)$ be a digraph where the minimum size of a dicut is $\tau$. Then there exists an assignment $\lambda_J\in \cZ$ to every strengthening set $J$ that intersects every minimum dicut exactly once, such that $\sum_{J}\lambda_J\1_J = \1$, $\1^\top \lambda = \tau$, and $\big\{\1_J:\lambda_J\neq 0\big\}$ is an integral basis for its linear hull.
\end{theorem}

Observe that \Cref{Woodall-CN-2} states that one can replace $\lambda_J \in \cZ$ by $\lambda_J \in \cZ_{\geq 0}$ in this theorem. This result does not extend to the capacitated setting; we shall explain this in \S\ref{subsec:schrijver-example}.

\begin{figure}[ht]
	\centering
	\includegraphics[scale=0.4]{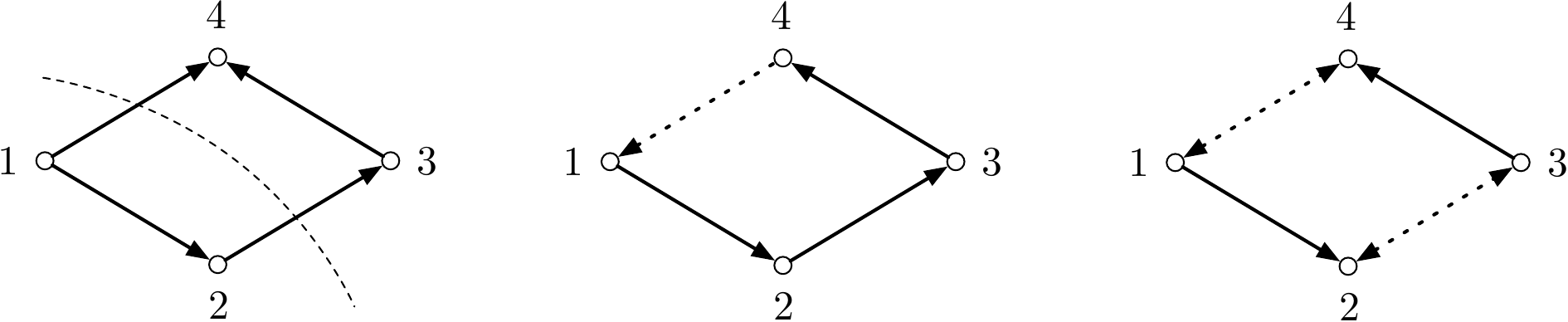}
	\caption{\b{A digraph $D=(V,A)$ (left). The set $\delta^+(\{1,2\})=\{(1,4),(2,3)\}$ is a dicut of $D$. The subset $\{(1,4)\}$ is a strengthening set of $D$, as reversing the arc in the set yields a strongly connected digraph (middle). The subset $\{(1,4),(2,3)\}$ is a dijoin of $D$, as bidirecting the arcs in the set yields a strongly connected digraph (right).}
	}
	\label{fig:dijoin-dicut}
\end{figure}

\paragraph{$p$-adic programming.} Given a prime number $p\geq 2$, a rational number is \emph{(finitely) $p$-adic} if it is of the form $r/p^k$ for some integer $r$ and nonnegative integer $k$, and a vector is \emph{$p$-adic} if each entry is a $p$-adic rational number. The $2$-adic, or \emph{dyadic}, rationals are important for numerical computations because they have a finite binary representation, and therefore can be represented exactly on a computer in floating-point arithmetic. Recently, the first two authors along with Guenin and Tun\c{c}el characterized when a linear program admits an optimal solution that is $p$-adic, and furthermore, they provided a polynomial algorithm for solving a linear program whose domain is restricted to the set of $p$-adic vectors~\cite{Abdi-DLP}. 

\Cref{scr-theorem} implies, for any prime number $p\geq 2$, the existence of a sparse $p$-adic optimal solution to a linear program related to packing dijoins. To elaborate, let $D=(V,A)$ be a digraph whose underlying undirected graph is connected. Denote by $M$ the matrix whose columns are labeled by $A$, and whose rows are the indicator vectors of the dijoins of $D$. Consider the following pair of dual linear programs,
$$
(P) \quad \min\left\{\1^\top x: Mx\geq \1, x\geq \0\right\} \qquad  \qquad
(D) \quad \max\left\{\1^\top y:M^\top y\leq \1, y\geq \0\right\}
$$ where $\1,\0$ denote the all-ones and all-zeros vectors of appropriate dimensions, respectively. A seminal theorem is that the primal linear program $(P)$ models exactly the \emph{minimum dicut problem}, i.e., $(P)$ admits an integral optimal solution~(\cite{Lucchesi78}, see~\cite{Cornuejols01}, \S1.3.4). Woodall's conjecture equivalently states that the dual linear program $(D)$, in turn, computes the maximum number of pairwise disjoint dijoins, that is, $(D)$ admits an integral optimal solution~\cite{Woodall78}. The main result of this paper implies some number-theoretic evidence for this conjecture, as it has the following consequence.

\begin{theorem}\label{p-adic-CO}
For any prime number $p\geq 2$, $(D)$ admits a $p$-adic optimal solution with at most $2|A|$ nonzero entries.
\end{theorem}

Observe that Carath\'{e}odory's theorem guarantees an optimal solution to $(D)$ with at most $|A|$ nonzero entries. \Cref{p-adic-CO} guarantees a $p$-adic optimal solution to $(D)$, all the while losing only a factor $2$ in the guarantee for the number of nonzero entries. 

This theorem does not extend to the capacitated setting. More specifically, if the objective function of $(P)$ is replaced by $c^\top x$ for a nonnegative integral vector $c$, then $(D)$ may not have a $p$-adic optimal solution, for any prime number $p\neq 2$, as we shall explain in \S\ref{subsec:schrijver-example}. Interestingly, it has very recently been shown that $(D)$ always admits a dyadic optimal solution in the capacitated setting~\cite{Guenin24+}; the techniques do not seem to yield a guarantee on the number of nonzero entries of a solution.

\paragraph{Hypergraph orientations.} Let $H=(V,\zE)$ be a hypergraph. An \emph{orientation of $H$} consists in designating to each hyperedge $E\in \zE$ a node inside as the \emph{head} of $E$, i.e., it is a mapping $O:\zE\to V$ such that $O(E)\in E$ for each $E\in \zE$. The orientation is \emph{strongly connected} if for each $X\subset V,X\neq \emptyset$, there exists a hyperedge whose designated head is inside $X$, and has at least one node outside $X$. 

Two orientations of $H$ are \emph{head-disjoint} if no hyperedge has the same head in both orientations. It is well-known that a graph, which is simply a $2$-uniform hypergraph, has $2$ head-disjoint strongly connected orientations if, and only if, the graph is $2$-edge-connected. The following unpublished conjecture by B\'{e}rczi and Chandrasekaran attempts to extend one direction of this to general $\tau$-uniform hypergraphs. 

Given two subsets $X,E\subseteq V$, we say that \emph{$X$ separates $E$} if $E\cap X\neq \emptyset$ and $E\not\subseteq X$. For $X\subseteq V$, denote by $d_H(X)$ the sum of $|X\cap E|$ ranging over all hyperedges $E\in \zE$ separated by $X$.

\begin{CN}\label{head-disjoint-CN}
Let $H=(V,\zE)$ be a $\tau$-uniform hypergraph such that $d_H(X)\geq \tau$ for all $X\subset V,X\neq \emptyset$. Then $H$ has $\tau$ pairwise head-disjoint strongly connected orientations. 
\end{CN}

For $\tau=3$, a weaker form of this conjecture appears explicitly in (\cite{Frank11}, Conjecture 9.4.15). We prove the following relaxation of this conjecture.

\begin{theorem}\label{head-disjoint-CO}
Let $\tau\geq 2$ be an integer, and let $H=(V,\zE)$ be a $\tau$-uniform hypergraph such that $d_H(X)\geq \tau$ for all $X\subset V,X\neq \emptyset$. Then there exists an assignment $\lambda_O\in \cZ$ to every strongly connected orientation $O:\zE\to V$ such that $$
 \sum \left(\lambda_O: \text{strongly connected orientation } O, O(E) = v\right) =1 \qquad \forall E\in \zE,\, \forall v\in E,
 $$ and $|\{O:\lambda_O\neq 0\}|\leq (\tau-1)|\zE|+1$.
\end{theorem}

Note that \Cref{head-disjoint-CN} states that \b{in \Cref{head-disjoint-CO}} one can replace $\lambda_O\in \cZ$ by $\lambda_O\in \cZ_{\geq 0}$.

\subsection{The dijoin polyhedron and digrafts}

\Cref{scr-theorem} is a consequence of a lattice theoretic result about the dijoin polyhedron of bipartite digraphs. To this end, for a digraph $D=(V,A)$, let $$\dij(D):=\left\{x\in \cR^A: x(\delta^+(U))\geq 1,\,\forall \text{ dicut $\delta^+(U)$}; x\geq \0\right\}.$$ It is known that $\dij(D)$ is an integral polyhedron, and its vertices are precisely the indicator vectors of the (inclusionwise) minimal dijoins of $D$~(\cite{Lucchesi78}, see \cite{Cornuejols01}, \S1.3.4). 

\begin{DE}[bipartite digraph]
A digraph is \emph{bipartite} if every node is a source or a sink.
\end{DE}

Recently, the first two authors and Zlatin demonstrated the importance of bipartite digraphs in studying Woodall's conjecture, by making steps towards the problem by first reducing the conjecture to a special class of bipartite digraphs~\cite{Abdi23-dijoins}. We shall follow these footsteps by studying faces of the dijoin polyhedron of a bipartite digraph.

\begin{DE}[digraft]
	A \emph{digraft} is a pair $(D=(V,A),\zF)$ where $D$ is a bipartite digraph, the underlying undirected graph of $D$ is $2$-edge-connected, and $\zF$ is a family over ground set $V$ such that (a) $\emptyset,V\notin \zF$, (b) if $U\in \zF$ then $\delta^-(U)=\emptyset$, (c) $V\setminus v\in \zF$ for every sink $v$ of $D$, and (d) the following face of $\dij(D)$ is nonempty: 
				$$F(D,\zF):=\dij(D)\cap \left\{x\in \cR^{A} : x(\delta^+(U))=1,\,\forall U\in \zF\right\}\subseteq [0,1]^A.$$ 
\end{DE}

\b{See \Cref{fig:digrafts} (left) for an illustration of a digraft.} The choice of the `digraft' terminology mirrors that of a `graft', an object that shows up in the context of the \emph{minimum $T$-join problem}, and is loosely related to the \emph{minimum dijoin problem} (see~\cite{Cornuejols01}, \S1.3.5).

\begin{figure}[ht]
	\centering
	\includegraphics[scale=0.4]{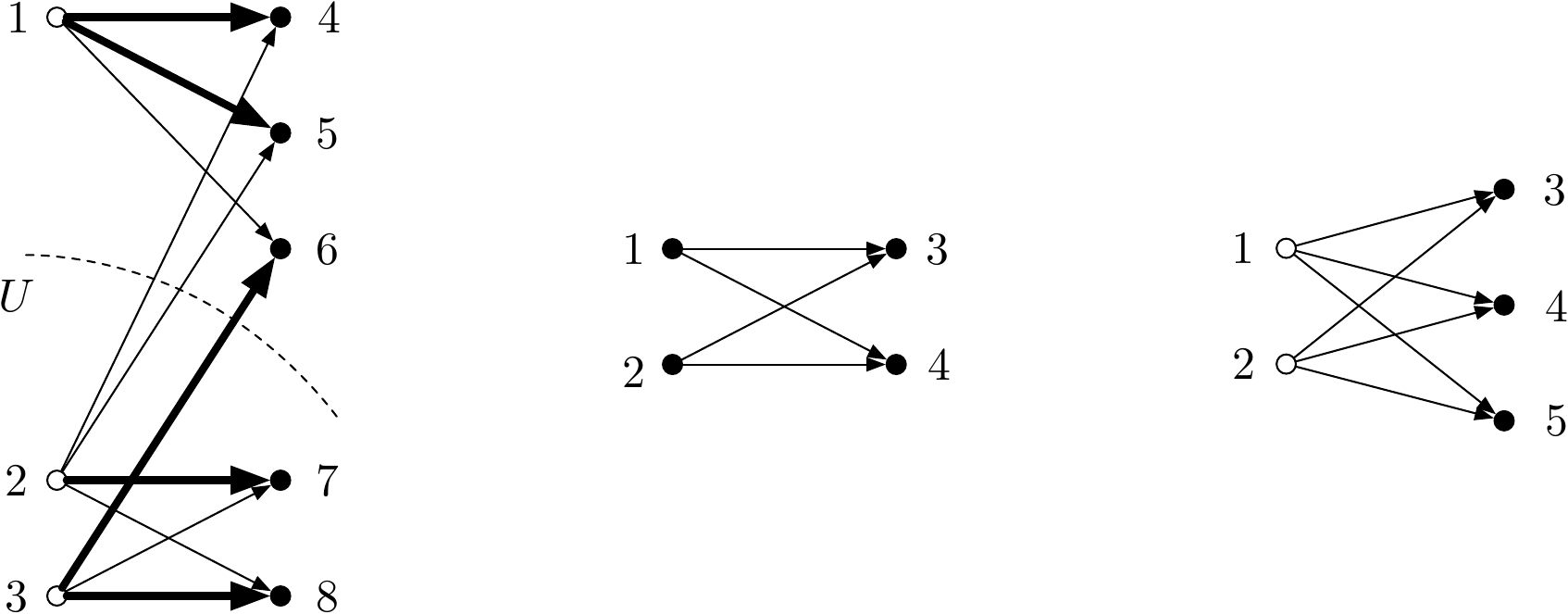}
\caption{\b{(Left) A digraft $(D,\zF)$ where $\zF$ consists of $V\setminus v$ for all sinks $v$, as well as $U=\{2,3,7,8\}$. The set of bold arcs is a dijoin $J$ whose indicator vector belongs to $F(D,\zF)$. (Middle/Right) A digraft $(D,\zF)$ where $\zF$ consists of $V\setminus v$ for all sinks $v$, and $\{u\}$ for every filled-in source $u$. The middle digraft is balanced basic, and robust. The right digraft is skewed basic, and robust. In all three digrafts, the filled-in nodes are tight, and the non-filled-in nodes are active.
}}
	\label{fig:digrafts}
\end{figure}

\Cref{scr-theorem} is a consequence of the following result.

\begin{theorem}\label{main-digraft}
	Let $(D=(V,A),\zF)$ be a digraft. Then $F(D,\zF)\cap \{0,1\}^A$ contains an integral basis for $\lin(F(D,\zF))$. 
\end{theorem}

It may not be clear how \Cref{main-digraft} is related to \Cref{scr-theorem}. To de-mystify this connection, let $(D=(V,A),\zF)$ be a digraft. As the underlying undirected graph of $D$ is $2$-edge-connected, every minimal dijoin is a strengthening set (see \cite{Schrijver03}, Theorem 55.1). Furthermore, every strengthening set that has exactly one arc incident with every sink, is also a minimal dijoin. \b{(For example, the bold arcs in 
\Cref{fig:digrafts} (left) form such a minimal dijoin.)} Subsequently, $F(D,\zF)=\scr(D)\cap \left\{x\in \cR^{A}:
x(\delta^+(U))-x(\delta^-(U))=1-|\delta^-(U)|,\,\forall U\in \zF
\right\}$. Furthermore, $\gcd\{1-|\delta^-(U)|: U\in \zF\}=1$. Thus, \Cref{main-digraft} follows from \Cref{scr-theorem}. 

The converse implication also holds, though we save the proof of this for \S\ref{sec:apps} after we prove \Cref{main-digraft} directly. \b{To shed some light on this, take a digraph $D=(V,A)$ and family $\zF$ from \Cref{scr-theorem}. The key idea is to transform the pair $D,\zF$ to a digraft $(D'=(V',A'),\zF')$, and in the process, bijectively map $F\cap \{0,1\}^A$ to $F(D',\zF')\cap \{0,1\}^{A'}$. In doing so, we map the proper nonempty node subsets $U$ of $D$ to node subsets $\varphi(U)$ of $D'$, where $\delta^-_{D'}(\varphi(U))=\emptyset$. See \Cref{fig:STR-to-dijoin} for an illustration of this, and \Cref{scr->dij} for further details.}

\begin{figure}[ht]
	\centering
	\includegraphics[scale=0.4]{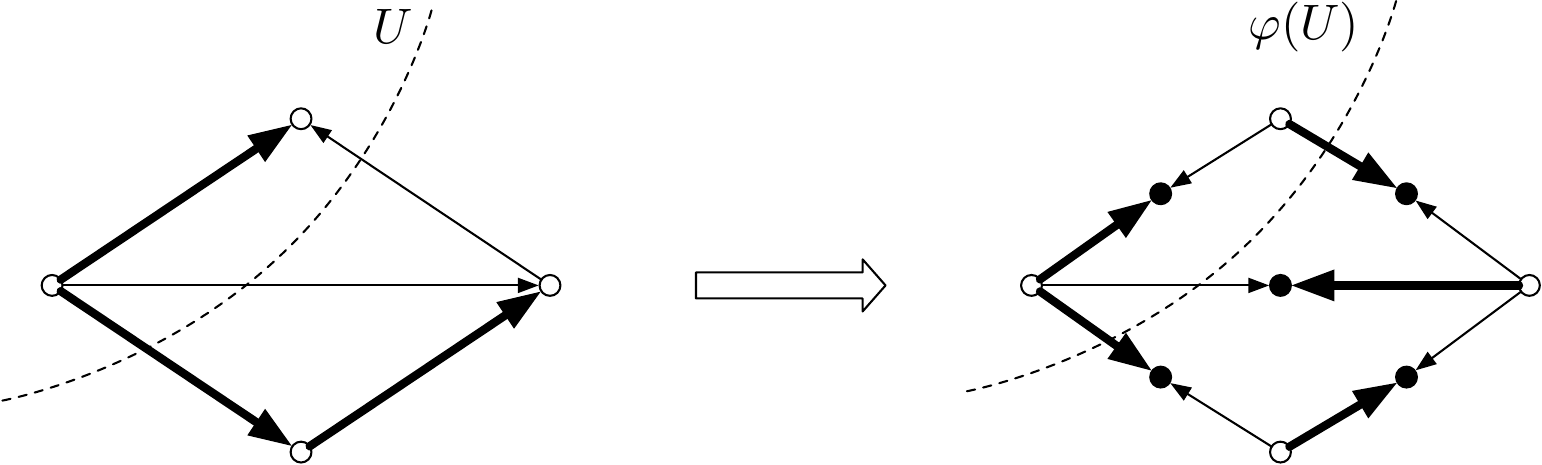}
\caption{\b{(Left) A digraph $D$. The set $J$ of bold arcs forms a strengthening set whose indicator vector belongs to $F$, for $\zF=\emptyset$. 
(Right) The digraft $(D',\zF')$ constructed in the proof of \Cref{scr->dij}. The filled-in nodes are new, while the non-filled-in ones are the original nodes. The set $\phi(J)$ of bold arcs is the image of $J$ under the mapping: it is a dijoin whose indicator vector belongs to $F(D',\zF')$.
}}
	\label{fig:STR-to-dijoin}
\end{figure}

We find \Cref{main-digraft} more convenient to work with than \Cref{scr-theorem}. At the highest level, one explanation for this is that every point in $F(D,\zF)\cap \{0,1\}^A$ is the indicator vector of an arc subset $J$ that has degree one at every sink of $D$, and degree at least one at every source of $D$; so that $J$ may be viewed as a perfect $b_J$-matching in a bipartite graph, for some degree vector $b_J$. Fixing the degree of $J$ at each sink to one has advantages: first, the cardinality of $J$ becomes invariant and equal to the number of sinks of $D$; secondly, in this case, $J$ is a minimal dijoin if and only if it is a strengthening set, an equivalence which we utilized above.

\subsection{Proof overview of \Cref{main-digraft}}

Let us provide an overview of the proof of \Cref{main-digraft}. To this end, let $(D=(V,A),\zF)$ be a digraft. Our goal is to find an integral basis in $F(D,\zF)\cap \{0,1\}^A$  for $\lin(F(D,\zF))$. The following is an important notion needed in the proof.
		
\begin{DE}[\b{balanced/skewed} basic digraft]
A digraft $(D=(V,A),\zF)$ is \emph{basic} if whenever $F(D,\zF)\subseteq \{x:x(\delta^+(U))=1\}$ for some dicut $\delta^+(U)$, then $U\in \zF$, and $|U|\in \{1,|V|-1\}$. \b{A \emph{balanced basic digraft} is a basic digraft with an equal number of sources and sinks. A \emph{skewed basic digraft} is a basic digraft with fewer sources than sinks.}
\end{DE}

The proof proceeds by first decomposing the digraft into basic pieces along dicuts $\delta^+(U)$ such that $1<|U|<|V|-1$ and $F(D,\zF)\subseteq \{x:x(\delta^+(U))=1\}$. Once we find integral bases for the basic pieces, then by composing the bases together in a natural manner, we obtain an integral basis in $F(D,\zF)\cap \{0,1\}^A$.

The challenge now is to find an integral basis for a basic digraft. Here comes a key idea of the proof, which is to study the facet-defining inequalities of $F(D,\zF)$.
	
\begin{DE}[basic robust digraft]
A basic digraft $(D=(V,A),\zF)$ is \emph{robust} if every facet-defining inequality for $F(D,\zF)$ is equivalent to $x_a\geq 0,a\in A$, or $x(\delta^+(u))\geq 1$ for some source $u$ of~$D$. 
\end{DE}

We then divide the proof into two cases, depending on whether the basic digraft is robust. 

\b{See \Cref{fig:digrafts} (middle/right) for illustrations of two relevant digrafts. In both instances, every minimal dicut is trivial, implying in turn that they are (balanced/skewed) basic robust digrafts.}

\b{See \Cref{fig:basic-non-robust-affine-critical} (left) for an illustration of a basic non-robust digraft. To see non-robustness, let $J\subset A$ be the set of bold arcs, and consider the dicut $\delta^+(U)$ for $U:=\{2,3,7,8\}$. Observe that $\1_J$ satisfies every (in)equality in the description of $F(D,\zF)$ except for $x(\delta^+(U))\geq 1$. This implies that the inequality $x(\delta^+(U))\geq 1$ is facet-defining for $F(D,\zF)$, but not equivalent to $x_a\geq 0,a\in A$, nor $x(\delta^+(u))\geq 1$ for any source $u$ of~$D$. 
}
	
\paragraph{Basic robust digrafts.}
For a basic robust digraft $(D,\zF)$, we prove that $F(D,\zF)$ is a very special polyhedron. To elaborate, for a polyhedron $P\subseteq \cR^n$ and $k\geq 0$, {let $kP:=\{kp:p\in P\}$}. \b{An integral polyhedron} $P$ has the \emph{integer decomposition property} if for every integer $k\geq 1$, every integral point in $kP$ can be written as the sum of $k$ integral points in $P$. The inequality description of $F(D,\zF)$, along with a classic result of de Werra~\cite{deWerra71} on balanced edge-colourings of bipartite graphs, allows for the following theorem.

\begin{theorem}\label{basic-robust-IDP}
Let $(D=(V,A),\zF)$ be a basic robust digraft. Then $F(D,\zF)$ has the integer decomposition property, and $\aff(F(D,\zF)) = \{x:Mx=\1\}$ for some $M\in \cZ^{m\times n}$ with $m\geq 1$.
\end{theorem}

\Cref{main-digraft} for basic robust digrafts now follows from the following general-purpose result about polyhedra with the integer decomposition property. 
    
\begin{theorem}\label{IDP-integral-basis}
Let $P\subseteq \cR^n$ be a pointed polyhedron with the integer decomposition property, where $\aff(P) = \{x:Ax=b\}$ for $A\in \cZ^{m\times n},b\in \cZ^m$ such that $m \geq 1$, $b\neq \0$, and $\gcd\{b_i:i\in [m]\}=1$. Then $P\cap \cZ^n$ contains an integral basis for $\lin(P)$.
\end{theorem}

\Cref{IDP-integral-basis} is obtained by first proving that $P\cap \cZ^n$ forms an \emph{integral generating set for a cone}, better known as a \emph{Hilbert basis}, and then using a result of Gerards and Seb\H{o}~\cite{Gerards87} about such sets to finish the proof.

\b{It must be mentioned that the GCD condition in \Cref{IDP-integral-basis} implies that `there exists $u\in \cZ^n$ such that $\aff(P)\subseteq \{x:u^\top x = 1\}$'. Furthermore, the converse of this implication also holds, in the sense that by forcing $u^\top x=1$ to be one of the rows of $Ax=b$, we can ensure that the GCD condition also holds. Some readers may find this condition easier to understand, as it is independent of any equality description of $\aff(P)$.} 

\b{Before we wrap the discussion on the integer decomposition property, it is worthwhile to mention that this is an important property in a number of areas of mathematics other than integer programming and the geometry of numbers, such as commutative algebra and algebraic geometry; see~\cite{Haase21} for a recent relevant article. For an introduction to the topic, we refer the reader to~\cite{Schrijver98}, \S22.10.}
	
\paragraph{Basic non-robust digrafts.} In the remaining case, where the basic digraft $(D=(V,A),\zF)$ is not robust, $F(D,\zF)$ has a facet-defining dicut inequality $x(\delta^+(U))\geq 1$ that is not equivalent to $x(\delta^+(u))\geq 1$ for any source $u$. We decompose the digraft into two pieces along the dicut $\delta^+(U)$, called the `$(U,V\setminus U)$-contractions' of $(D,\zF)$. Each of the two $(U,V\setminus U)$-contractions is again a digraft, so by induction, we may pick integral bases $B_1,B_2$ for the two pieces, and compose them in a natural way to obtain a linearly independent set $B'\subseteq F(D,\zF)\cap \{0,1\}^A$. However, there are two key challenges to turn $B'$ into an integral basis $B$. First, $B'$ is at least one vector away from forming a linear basis for $\lin(F(D,\zF))$, and a priori, we do not know the number of extra vectors we would need to add. Secondly, a linear basis is a long way from an integral one, so we need to extend $B'$ very carefully. We have two lemmas that address these issues.

\begin{figure}[ht]
\centering
	\includegraphics[scale=0.4]{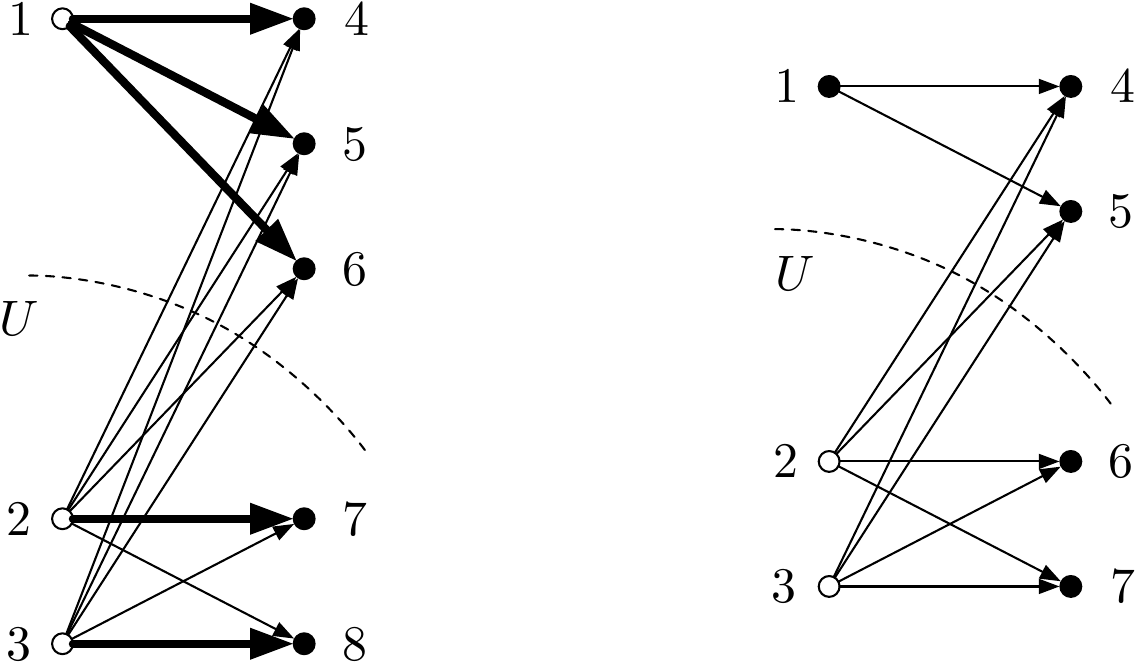}
	\caption{\b{(Left) 
	A basic non-robust digraft $(D,\zF)$ where $\zF$ consists only of $V\setminus v$ for all sinks $v$. 
	(Right) 
	A affine critical digraft $(D,\zF)$ where $\zF$ consists only of $V\setminus v$ for all sinks $v$. This digraft is not basic. 
	 In both figures, the filled-in/non-filled-in correspond to the tight/active nodes.
	} 
		}
	\label{fig:basic-non-robust-affine-critical}
\end{figure}

The first issue stems from the fact that the two $(U,V\setminus U)$-contractions are not necessarily basic digrafts. This can be fatal as we could lose our guarantee on the size of $|B\setminus B'|$. However, we will be able to prove that both of these pieces share a key property with basic digrafts, thus allowing us to guarantee that $|B\setminus B'|=1$. To describe this property, we need a couple of definitions.
	
\begin{DE}[tight and active nodes]
Let $(D=(V,A),\zF)$ be a digraft. A node $u\in V$ is \emph{tight} for the digraft if $F(D,\zF)\subseteq \{x:x(\delta(v))=1\}$; the node is \emph{active} for the digraft if it is not tight.
\end{DE}
	
Note that all sinks of a digraft are tight, i.e., every active node is a source. Note further that while singletons and complements of singletons in $\zF$ give rise to tight nodes, there may be more (implied) tight nodes. \b{In \Cref{fig:digrafts}, the filled-in/non-filled-in nodes correspond to the tight/active nodes, respectively. Observe that for the middle digraft, we may drop $\{1\},\{2\}$ from $\zF$ without changing $F(D,\zF)$; in doing so, the digraft would no longer be basic, as $1,2$ would be (implied) tight nodes not from $\zF$.
}
	
\begin{DE}[affine critical digraft]
Let $(D=(V,A),\zF)$ be a digraft, and let $V^t$ be the set of tight nodes. The digraft is \emph{affine critical} if $\aff(F(D,\zF))=\big\{x:x(\delta(v))=1,\,\forall v\in V^t\big\}$.
\end{DE}

\b{Every basic digraft is clearly affine critical. See \Cref{fig:basic-non-robust-affine-critical} (right) for an illustration of an affine critical digraft that is not basic. To see this, note that for $U:=\{2,3,6,7\}$, $\delta^+(U)$ is the only non-trivial minimal dicut. For all $x\in F(D,\zF)$, as $x(\delta(4))=x(\delta(5))=1$, and $x(\delta(1)),x(\delta^+(U))\geq 1$, it follows that $x(\delta(1))=x(\delta^+(U))=1$. Subsequently, this digraft is not basic. However, the digraft is affine critical, because $x(\delta^+(U))=1$ is implied by tight nodes: $$x(\delta^+(U)) = x(\delta(4))+x(\delta(5))-x(\delta(1))=1.$$}

The following important lemma addresses the first issue mentioned above.

\begin{LE}[Affine Critical Lemma]\label{affine-critical-LE}
Let $(D=(V,A),\zF)$ be a basic digraft that is not robust. Let $x(\delta^+(U))\geq 1$ be a facet-defining dicut inequality for $F(D,\zF)$ that is not equivalent to $x(\delta^+(u))\geq 1$ for any active source~$u$. Then $(D,\zF)$ and its $(U,V\setminus U)$-contractions are affine critical digrafts each of which contains at least one active source. Furthermore, for $i\in \{1,2\}$, every active source for $(D,\zF)$ in $U_{3-i}$ is an active source for $(D_i,\zF_i)$, and vice versa.
\end{LE}

This lemma is a byproduct of a careful analysis of the dimension of $F(D,\zF)$, and the study of a characteristic quantity of a digraft that we call the \emph{slack}.\medskip

The Affine Critical Lemma guarantees that to turn $B'$ into a linear basis for $\lin(F(D,\zF))$, we just need to add one more vector $b$ from $F(D,\zF)\cap \{0,1\}^A$, which must inevitably satisfy $b(\delta^+(U))>1$. As it turns out, integrality of the basis can be guaranteed if $b(\delta^+(U))=2$, whose existence will be guaranteed by the following lemma, thus addressing the second issue. This lemma is ultimately enabled by the Exchange Axiom for $M$-convex sets.

\begin{LE}[Jump-Free Lemma]\label{jump-free-LE}
Let $(D=(V,A),\zF)$ be a digraft, let $\delta^+(U)$ be a dicut, and let $x,y\in F(D,\zF)\cap \{0,1\}^A$ where $\lambda_1:=x(\delta^+(U))<y(\delta^+(U))=:\lambda_2$. Then for any integer $\lambda\in(\lambda_1,\lambda_2)$, there exists $z\in F(D,\zF)\cap \{0,1\}^A$ such that $z(\delta^+(U))=\lambda$.
\end{LE}

\paragraph{Outline of the paper.} We start off in \S\ref{sec:IGSC-IDP} by studying the integer decomposition property, and proving \Cref{IDP-integral-basis}. In \S\ref{sec:jump-free-lemma}, we discuss the exchange axiom for $M$-convex sets, and use it to prove the Jump-Free Lemma. In \S\ref{sec:dimension-dijoins-faces}, we introduce the `slack' of a digraft, study the dimension of the faces of the dijoin polyhedron, and then prove the Affine Critical Lemma. These three sections can be read independently from one another. \Cref{IDP-integral-basis} and \Cref{main-digraft} are then proved in \S\ref{sec:main-proof}. The final section \S\ref{sec:apps} is dedicated to proving \Cref{scr-theorem} and its three applications; \b{it also includes a conjecture on extending this theorem to all faces of $\scr(D)$.}

As a last note, it should be acknowledged that our work is heavily inspired by the works of Edmonds, Lov\'{a}sz, and Pulleyblank~\cite{Edmonds82}, Lov\'{a}sz~\cite{Lovasz87}, and  Carvalho, Lucchesi, and Murty~\cite{Carvalho02} on the matching lattice of a matching-covered graph. 

\section{IGSCs and the integer decomposition property}\label{sec:IGSC-IDP}

A finite set $H\subseteq \cZ^n$ is an \emph{integral generating set for a subspace (IGSS)} if every integral vector in $\lin(H)$ can be written as an integer linear combination of the vectors in $H$. A finite set $H\subseteq \cZ^n$ is an \emph{integral generating set for a cone (IGSC)} if every integral vector in $\cone(H)$ can be written as an integral conic combination of the vectors in $H$ (\cite{Abdi24-TDD}, \S7).\footnote{Sometimes $H$ is referred to as a \emph{Hilbert basis}, but we refrain from using this terminology as it can be confusing.} It can be readily checked that every IGSC is also an IGSS (see \cite{Abdi24-TDD}, \S7). We have the following theorem, which is essentially due to Gerards and Seb\H{o}~\cite{Gerards87}.
 
\begin{theorem}\label{IGSC-integral-basis}
Let $H\subseteq \cZ^n$ be an IGSC such that $\cone(H)$ is pointed. Then $H$ contains an integral basis for $\lin(H)$.
\end{theorem}
\begin{proof}
Let $H\subseteq \cZ^n$ be an IGSC such that $\cone(H)$ is pointed. Let $r$ be the rank of $H$. It follows from (\cite{Gerards87}, (2)) that there exist linearly independent vectors $h_1,\ldots,h_r$ in $H$ such that $B:=\{h_1,\ldots,h_r\}$ is an IGSC, hence an IGSS. Thus, $B$ is an integral basis, as required.
\end{proof}

Note that the condition that $\cone(H)$ is pointed is necessary. For example, $\{-2,3\}$ forms an IGSC, but it does not contain an integral basis for~$\cR$.

Let $g\in \cZ_{\geq 1}$. A finite set $H\subseteq \cZ^n$ is a \emph{$\frac{1}{g}$-integral generating set for a cone ($\frac{1}{g}$-IGSC)} if every integral vector in $\cone(H)$ can be written as a $\frac{1}{g}$-integral conic combination of the vectors in $H$. 

\begin{theorem}\label{IDP-gIGSC}
Let $P\subseteq \cR^n$ be a polyhedron with the integer decomposition property, where $\aff(P) = \{x:Ax=b\}$ for some $A\in \cZ^{m\times n},b\in \cZ^m\setminus \0$ with $m\geq 1$. Let $g:=\gcd\{b_i:i\in [m]\}$. Then $P\cap \cZ^n$ is a $\frac{1}{g}$-IGSC. 
\end{theorem}
\begin{proof}
Let $\bar{x}$ be an integral vector in $\cone(P\cap \cZ^n)$, so $\bar{x} = \sum_{p\in P\cap \cZ^n} \lambda_p p$ for some assignment $\lambda_p\in \cR_{\geq 0}$ to every $p\in P\cap \cZ^n$, where only finitely many $\lambda_p$'s are nonzero. Let $k:=\1^\top \lambda\geq 0$. Note that $\bar{x}\in kP\cap \cZ^n$. We claim that $gk\in \cZ$. To see this, note that $$A\bar{x} = \sum_{p\in P\cap \cZ^n} \lambda_p A p = \sum_{p\in P\cap \cZ^n} \lambda_p b= k b.$$ Given that both $A,\bar{x}$ are integral, it follows that $A\bar{x} = k b$ is also integral, so $k\in \frac{1}{g}\cZ$.

If $k=0$, then $\bar{x}=0$. Otherwise, $gk\geq 1$. As $g\bar{x}\in gkP\cap \cZ^n$ and $P$ has the integer decomposition property, it follows that $g\bar{x}$ can be written as the sum of $gk$ points in $P\cap \cZ^n$. In both cases, we expressed $\bar{x}$ as a $\frac{1}{g}$-integral conic combination of the vectors in $P\cap \cZ^n$, thus finishing the proof.
\end{proof}

We are now ready to prove \Cref{IDP-integral-basis}.

\begin{proof}[Proof of \Cref{IDP-integral-basis}]
Let $P\subseteq \cR^n$ be a pointed polyhedron with the integer decomposition property, where $\aff(P) = \{x:Ax=b\}$ for $A\in \cZ^{m\times n},b\in \cZ^m$ such that $m \geq 1$, $b\neq \0$, and $\gcd\{b_i:i\in [m]\}=1$. It follows from \Cref{IDP-gIGSC} for $g=1$ that $P\cap \cZ^n$ is an IGSC. As $P$ is pointed, so is $\cone(P\cap \cZ^n)$, so by \Cref{IGSC-integral-basis}, $P\cap \cZ^n$ contains an integral basis for $\lin(P\cap \cZ^n)=\lin(P)$, where this equality follows from the integrality of $P$, as required.
\end{proof}

\section{$M$-convex sets and the Jump-Free Lemma}\label{sec:jump-free-lemma}

Let $(D=(V,A),\zF)$ be a digraft. In this section, we will see an affine function which maps $F(D,\zF)\subseteq \cR^A$ to a base polyhedron in $\cR^V$, and $F(D,\zF)\cap \{0,1\}^A$ to an `$M$-convex set'.\footnote{The terminology in this section relating to discrete convex analysis follows \cite{Murota03}.} The theory of perfect $b$-matchings in bipartite graphs allows us to construct a (not necessarily unique) inverse to this function. This inverse map, together with the `Exchange Axiom' for $M$-convex sets leads to a proof of the Jump-Free Lemma. We will also state and prove a result, more specifically \Cref{arc-active}, which will be needed in the next section. \b{Moving forward, we shall need the following key concept.}

\b{\begin{DE}[crossing family, crossing pair]
Let $\zU$ of subsets of a finite ground set $V$. Then $\zU$ is a \emph{crossing family} if $U\cap W,U\cup W\in \zU$ for all pairs $U,W\in \zU$ that \emph{cross}, i.e., $U\cap W\neq \emptyset$, $U\cup W\neq V$, $U\setminus W,W\setminus V\neq \emptyset$.
\end{DE}}

For a node subset $U\subseteq V$, denote by $\sources(U)$ and $\sinks(U)$ the sets of sources and sinks in $U$, respectively, and define the \emph{discrepancy of $U$} as $\disc(U):=|\sinks(U)|-|\sources(U)|$~\cite{Abdi23-dijoins}. 
Let $\zU:=\{U\subset V:U\neq \emptyset, \delta^-(U)=\emptyset\}$, which is a {crossing family over ground set $V$}. The function $\disc:\zU\to \cZ$ forms a \emph{crossing supermodular function}, meaning that $\disc(U\cap W)+\disc(U\cup W)\geq \disc(U)+\disc(W)$ for all pairs $U,W\in \zU$ that cross. (In fact, equality holds here, but all we need is the inequality.)

Consider the polytope 
$$P(D):=\left\{z\in \cR^{V}:
z(U)\geq 1+\disc(U),\,\forall U\in \zU;\,
z(V) =\disc(V)\right\}.$$ 
Above, we have an inequality for every set $U$ in the crossing family $\zU$, and an equality constraint for the ground set $V$. Given that the right-hand side values of the inequalities $U\mapsto 1+\disc(U)$ form a crossing supermodular function over the crossing family $\zU$, and $P(D)\neq \emptyset$ {by \Cref{degree-vector-LE} below}, we thus obtain that $P(D)$ is an integral \emph{base polyhedron} (a.k.a.\ a polymatroid), by a result of Fujishige~\cite{Fujishige84}. {
For every $z\in P(D)$, source $u$ and sink $v$ of $D$, we have $z_u\geq 0\geq z_v$. This holds because $\{u\},V\setminus \{v\}\in \zU$, so $z_u\geq 1+\disc(\{u\})=0$, and $$z(V)-z_v=z(V\setminus \{v\}) \geq 1+\disc(V\setminus \{v\})=1+\disc(V)-\disc(\{v\}) = 1+ z(V) -1,$$ implying in turn that $0\geq z_v$.} 

Let $$P(D,\zF) := P(D)\cap \{z: z_v=0,\,\forall v\in \sinks(V);\, z(U) =1+\disc(U),\,\forall U\in \zF\}.$$ 
As a face of an integral base polyhedron, $P(D,\zF)$ is also an integral base polyhedron. Subsequently, $P(D,\zF)\cap \cZ^V$ is an \emph{$M$-convex set}, that is, it possesses the following property: \begin{quote}
	\emph{Exchange Axiom:} For $z,t\in P(D,\zF)\cap \cZ^V$ and $u\in \supp^+(z-t)$, there exists $v\in \supp^-(z-t)$ such that $z':=z-\1_u+\1_v\in P(D,\zF)\cap \cZ^V$. We say that $z'$ is obtained by the exchange pair $(u,v)$ for $(z,t)$.
\end{quote} Above, $\supp^+(z) = \{v:z_v>0\}$ and $\supp^-(z) = \{v:z_v<0\}$. For more on base polyhedra and $M$-convex sets, we refer the reader to Murota's excellent book on Discrete Convex Analysis (\cite{Murota03}, Chapter~4). \b{Recently, $M$-convex sets have been characterized by the so-called `Lorentzian property', see~\cite{Branden20} for more.}

{
\begin{LE}\label{degree-vector-LE}
Let $x\in \cR^A$ such that $x(\delta(v)) = 1$ for every sink $v$. For each $v\in V$, let $z_v:=x(\delta(v)) - 1$. Then $x\in F(D,\zF)$ if and only if $z\in P(D,\zF)$. Furthermore, $P(D)\neq \emptyset$.
\end{LE}
\begin{proof}
Clearly, $z_v=0$ for every sink $v$. Furthermore, for every $U\subseteq V$ such that $\delta^-(U)=\emptyset$, we have
\begin{align*}
z(U) 
&= z(\sources(U))\\ 
&= \sum_{u\in \sources(U)} x(\delta^+(u)) - |\sources(U)| \\
&= \sum_{u\in \sources(U)} x(\delta^+(u)) - |\sources(U)| - \sum_{v\in \sinks(U)} x(\delta^-(v)) + |\sinks(U)|\\
& = x(\delta^+(U))-x(\delta^-(U)+\disc(U)\\
& = x(\delta^+(U))+\disc(U).
\end{align*} A consequence is that $
z(V) = \disc(V)$. Another is that for all $U\in \zU$, $z(U)\geq 1+\disc(U)$ if and only if $x(\delta^+(U))\geq 1$, with the slacks in both inequalities being equal. Subsequently, $x\in F(D,\zF)$ if and only if $z\in P(D,\zF)$. In particular, given that $F(D,\zF)\neq \emptyset$, we conclude that $P(D,\zF)\neq \emptyset$, and so $P(D)\neq \emptyset$.
\end{proof}
}

We just showed how to map every (integral) point in $F(D,\zF)$ to an (integral) point in $P(D,\zF)$, \b{namely its `degree vector' reduced by the all-ones vector}. Below we show that it is possible to go in reverse; only the first part of the lemma is needed for the proof of the Jump-Free Lemma, while the second (stronger) part is needed later.

\begin{LE}\label{perfect-b-matching-LE}
	Let $(D=(V,A),\zF)$ be a digraft, and let $z\in P(D,\zF)\cap \cZ^V$. \begin{enumerate}
		\item There exists $J\subseteq A$ such that $|J\cap \delta(v)|-1=z_v$ for each $v\in V$. 
		\item For each $a\in A$, there exists $J\subseteq A$ such that $a\in J$ and $|J\cap \delta(v)|-1=z_v$ for each $v\in V$.
	\end{enumerate} 
\end{LE}
\begin{proof}
Let $b:=\1+z\in \cZ^V$. Note that $b_v=1$ for every sink $v$, and $b_u\geq 1$ for every source $u$.
	 
	 {\bf (1)} This part asks for an arc subset $J$ such that $|J\cap \delta(v)|=b_v$. A \emph{perfect $b$-matching} is a vector $x\in \cZ_{\geq 0}^A$ such that $x(\delta(v))=b_v$ for every node $v$. Given that $b_v=1$ for every sink $v$, and $x\geq \0$, it follows that every perfect $b$-matching is a $0,1$ vector. Thus, to prove this part, it suffices to argue the existence of a perfect $b$-matching.
	
To this end, denote by $S$ and $T$ the sets of sources and sinks of $D$, respectively. It is known that a perfect $b$-matching exists if, and only if, $b(S)=b(T)$, and for $b(U\cap S)-b(U\cap T)\geq 0$ for every dicut $\delta^+(U)$ (\cite{Schrijver03}, Corollary 21.1b, see \cite{Abdi23-dijoins}, Theorem 4.10).

Given that $b_v=1$ for every sink $v$, the equality $z(V)=\disc(V)$ is equivalent to $b(S) = b(T)$, while the inequality $z(U)\geq 1+\disc(U)$ is equivalent to $b(U\cap S)-b(U\cap T)\geq 1$ for every dicut $\delta^+(U)$. In particular, there exists a perfect $b$-matching, as required.
	
	{\bf (2)} Suppose $a$ has tail $u$ and head $w$. Let $D':=D\setminus w$, define $z'\in \cZ^{V\setminus w}$ as $z'_u:=z_u-1$ and $z'_v:=z_v$ for all $v\in V\setminus w\setminus u$, and let $b':=\1+z'\in \cZ^{V\setminus w}_{\geq 0}$. We claim that $D'$ has a perfect $b'$-matching. To this end, note first that $b'(S) = b(S)-1 =  b(T\setminus w)=b'(T\setminus w)$. Furthermore, let $\delta^+_{D'}(U)$ be a dicut of $D'$ for some $U\subset V\setminus w, U\neq \emptyset$. Observe that $\delta^+_D(U)$ is a dicut of $D$, so by the previous part, $b(U\cap S)-b(U\cap T)\geq 1$. Subsequently, $$b'(U\cap S)-b'(U\cap (T\setminus w))\geq (b(U\cap S)-1)-b'(U\cap (T\setminus w))  = b(U\cap S)-1-b(U\cap T)\geq 0.$$ As this inequality holds for every dicut $\delta^+_{D'}(U)$, $D'$ has a perfect $b'$-matching, as claimed. Adding arc $a$, we obtain a perfect $b$-matching $x$ in $D$ such that $x_a=1$, thereby proving this part.
	\end{proof}

We are now ready to prove the Jump-Free Lemma.

\begin{proof}[Proof of \Cref{jump-free-LE}]
	Let $(D=(V,A),\zF)$ be a digraft, let $\delta^+(U)$ be a dicut, and let $J_1,J_2\subseteq A$ be subsets such that $\1_{J_1},\1_{J_2}\in F(D,\zF)$ and $\lambda_1:=|J_1\cap \delta^+(U)|<|J_2\cap \delta^+(U)|=:\lambda_2$. Our goal is to prove that for any integer $\lambda\in(\lambda_1,\lambda_2)$, there exists a $J\subseteq A$ such that $\1_J\in F(D,\zF)$ and $|J\cap \delta^+(U)|=\lambda$.
	
	For each $i\in \{1,2\}$, and node $u$, let $z^i_u:=|J_i\cap \delta(u)|-1$. \b{That is, $z^1,z^2$ are the degree vectors of $
	\1_{J_1},\1_{J_2}\in F(D,\zF)$ reduced by the all-ones vector, respectively.} {Thus} $z^1,z^2\in P(D,\zF)$ \b{by \Cref{degree-vector-LE}}. By the Exchange Axiom for the $M$-convex set $P(D,\zF)\cap \cZ^V$, there exists a sequence of points $z^1=:t^1,t^2,\ldots,t^k:=z^2$ in $P(D,\zF)\cap \cZ^V$ such that $t^{i+1}$ is obtained by an exchange pair $(u^i,v^i)$ for $(t^i,t^k)$, for each $1\leq i\leq k-1$. Note that $k-1$ is precisely the sum of the nonnegative entries of $z^1-z^2$. Note further that for each $1\leq i\leq k-1$, $$
	t^i(U)-t^{i+1}(U) = \1_{u_i}(U)-\1_{v_i}(U)\in \{-1,0,1\},
	$$ thus the set $\{t^i(U):i=1,\ldots,k\}$ contains all the integers between $z^1(U)$ and $z^2(U)$. Subsequently, given that $$
	z^1(U)-\disc(U)=\lambda_1<\lambda_2=z^2(U)-\disc(U),
	$$ we have $$\{t^i(U)-\disc(U):i=1,\ldots,k\}\supseteq [\lambda_1,\lambda_2]\cap \cZ.$$ Now pick an integer $\lambda\in (\lambda_1,\lambda_2)$, and a $t^i$ such that $t^i(U)-\disc(U)=\lambda$. By \Cref{perfect-b-matching-LE}~(1), there exists a $J\subseteq A$ such that $|J\cap \delta(v)|-1=t^i_v$ for each $v\in V$, and so in particular, $\1_J\in F(D,\zF)$. \b{The inclusion follows from \Cref{degree-vector-LE}, as the degree vector of $\1_J$ reduced by the all-ones vector is precisely $t^i$, which belongs to $P(D,\zF)$.} This is the desired set $J$, because $|J\cap \delta^+(U)|=t^i(U)-\disc(U)=\lambda$.
\end{proof}

\Cref{perfect-b-matching-LE}~(2) has the following consequence, which will be useful in the next section.

\begin{theorem}\label{arc-active}
	Let $(D=(V,A),\zF)$ be a digraft. Then, for every $a\in A$, there exists $J\subseteq A$ such that $a\in J$ and $\1_J\in F(D,\zF)$. 
\end{theorem}
\begin{proof}
Pick an arbitrary point $z\in P(D,\zF)\cap \cZ^V$, and let $a\in A$. Then by \Cref{perfect-b-matching-LE}~(2), there exists $J\subseteq A$ such that $a\in J$ and $|J\cap \delta(v)|-1=z_v$ for each $v\in V$. As $z\in P(D,\zF)$, it follows that $\1_J\in F(D,\zF)$, as required.
\end{proof}

\section{The slack, dicut contraction, and the Affine Critical Lemma}\label{sec:dimension-dijoins-faces}

In the introduction, we motivated digrafts as a means to describe faces of $\dij(D)$ for a bipartite digraph $D$ obtained by setting some dicut inequalities to equality. In this section, we first count the dimension of the face $F(D,\zF)$ for a digraft $(D,\zF)$, and in the process, define the notion of the `slack', a novel and characteristic quantity associated with a digraft. When counting the dimension, we must study the equations that define the affine hull of $F(D,\zF)$. It is not clear whether some non-negativity constraints used to define $\dij(D)$, are forced to equality in $F(D,\zF)$. The following immediate corollary of \Cref{arc-active} shows that this is fortunately not the case.

\begin{CO}\label{arc-active-CO}
$F(D,\zF)\not\subseteq \{x:x_a=0\}$ for any digraft $(D=(V,A),\zF)$ and any arc $a\in A$.\qed
\end{CO}

Therefore, the affine hull of $F(D,\zF)$ can be described solely by setting some dicut inequalities to equality. Roughly speaking, the `slack' captures the contribution of the `non-trivial dicut' inequalities in defining the affine hull. 

We shall formalize the notion of the slack in \S\ref{subsec:slack}. In \S\ref{subsec:de-composition}, we show how to `decompose' a digraft along a `contractible' dicut. In \S\ref{subsec:FDI}, we study for a basic digraft the facet-defining inequalities of $F(D,\zF)$ corresponding to contractible dicuts, and prove crucially that decomposition along this dicut does not change the slack. We are then well-equipped to prove the Affine Critical Lemma in \S\ref{subsec:affine-critical-LE}.

\subsection{The slack}\label{subsec:slack}

Let $(D=(V,A),\zF)$ be a digraft, and let $V^t$ be the set of tight nodes. Let $F:=F(D,\zF)$. We wish to lower bound the rank $r$ of the equations that hold for $F$. Let $\aff(F)$ be the affine hull of $F$. By \Cref{arc-active-CO}, $\aff(F)$ can be described by two types of equations: (i) $x(\delta(v))=1\,\forall v\in V^t$, and (ii) $x(\delta^+(U))=1$ for some dicuts $\delta^+(U)$ where $|U|\neq 1,|V|-1$. To lower bound $r$, we need the following remark.

\begin{RE}\label{linear-independence-RE}
	Let $G=(V,E)$ be a connected graph, and let $v^\star\in V$. Then $\1_{\delta(u)},u\in V\setminus v^\star$ are linearly independent.
\end{RE}

It can be readily checked that \begin{equation}\label{eq:rank-ineq}
	r \geq |V^t|-1,
\end{equation} where we have used the fact that the equations $x(\delta(u))=1,\,\forall u\in V^t$ have rank at least $|V^t|-1$, by \Cref{linear-independence-RE}. Furthermore, as the underlying undirected graph of $D$ is bipartite, these equations have rank exactly $|V^t|-1$ if and only if $V=V^t$. 

\begin{DE}[$\kappa^t$]\label{def:kappa-t}
Denote by $\kappa^t(D,\zF)$ the indicator variable for the event $V=V^t$, set to $1$ if the event occurs, and $0$ otherwise.
\end{DE}

As discussed above, the inequality \eqref{eq:rank-ineq} can be strengthened as follows: \begin{equation}\label{eq:rank-ineq-strengthened}
	r \geq |V^t|-\kappa^t(D,\zF),
\end{equation} 
where we have subtracted a $1$ only if every node of $(D,\zF)$ is tight.

Denote by $\dim(F)$ the dimension of $\aff(F)$. Then $\dim(F) = |A| - r$, so we obtain from \eqref{eq:rank-ineq-strengthened} that \begin{equation}\label{eq:dim-ineq}
	|A|-|V^t|+\kappa^t(D,\zF)\geq \dim(F).
\end{equation}
There is a subtle difference between the dimensions of the affine and linear hulls of $F$. Note that the dimension of the linear hull of $F$ is $\dim(F)+1$, given that $\0\notin \aff(F)$.

A characteristic quantity associated with a digraft is the slack in \eqref{eq:dim-ineq}.

\begin{DE}[slack]\label{def:slack}
The \emph{slack} of $(D=(V,A),\zF)$ is $$
s(D,\zF):=|A|-|V^t|+\kappa^t(D,\zF)- \dim(F(D,\zF))\geq 0,
$$ where $V^t$ denotes the set of tight nodes of $(D,\zF)$.
\end{DE}

The following is a straightforward but important characterization of when there is no slack.

\begin{RE}\label{slack=0}
$s(D,\zF)=0$ if and only if $\aff(F(D,\zF))$ is described by $x(\delta(v))=1,\,v\in V^t$.
\end{RE}

Subsequently, if $s(D,\zF)\geq 1$, then the description of $\aff(F(D,\zF))$ also sets some `non-trivial' dicut inequalities to equality.

\begin{DE}[tight dicut]
	A dicut $\delta^+(U)$ is \emph{tight} for $(D=(V,A),\zF)$ if $F(D,\zF)\subseteq \{x:x(\delta^+(U))=1\}$. A tight dicut $\delta^+(U)$ is \emph{non-trivial} if $1<|U|<|V|-1$, otherwise it is \emph{trivial}.
\end{DE}

The following theorem is the main result of this subsection.

\begin{theorem}\label{slack>=1}
Let $(D=(V,A),\zF)$ be a digraft. Then $s(D,\zF)\geq 1$ if, and only if, there exists a non-trivial tight dicut $\delta^+(U)$ where both $U,V\setminus U$ contain active sources.
\end{theorem}
\begin{proof}
Denote by $V^t,V^a$ the sets of tight and active nodes of $(D,\zF)$. 
	
$(\Leftarrow)$ Suppose $\delta^+(U)$ is a non-trivial tight dicut of the digraft such that both $U,V\setminus U$ contain active sources. By definition, the points in $F(D,\zF)$ satisfy the equation $x(\delta^+(U))=1$. We claim that this equation is not implied by the equations $x(\delta(v)) = 1,~\forall v\in V^t$, which were used to obtain \eqref{eq:rank-ineq-strengthened}, thus giving an improvement of $1$ to that inequality, which eventually implies $s(D,\zF)\geq 1$. 
	
	To prove the linear independence of $x(\delta^+(U))=1$ from the other equations, it suffices to prove that $\1_{\delta^+(U)}$ is linearly independent of the vectors $\1_{\delta(v)}, v\in V^t$. Note that $$
	 \1_{\delta^+(U)} = \sum_{u\in \sources(U)} \1_{\delta^+(u)} - \sum_{v\in \sinks(U)} \1_{\delta^-(v)}.$$ 
	After subtracting a linear combination of $\1_{\delta(v)}, v\in V^t$ from $
	 \1_{\delta^+(U)}$, it suffices to prove that the vector $\sum_{u\in \sources(U)\cap V^a} \1_{\delta(u)}$ is linearly independent of $\1_{\delta(v)}, v\in V^t$. Given that both $U,V\setminus U$ contain active sources, it follows that $\sources(U)\cap V^a$ is a nonempty proper subset of $V^a$. Let $v^\star\in \sources(V\setminus U)\cap V^a$. By \Cref{linear-independence-RE}, the vectors $\1_{\delta(v)},v\in V\setminus v^\star$ are linearly independent, implying in turn that $\sum_{u\in \sources(U)\cap V^a} \1_{\delta(u)}$ is linearly independent of $\1_{\delta(v)}, v\in V^t$.	 
	 
	$(\Rightarrow)$ Suppose $s(D,\zF)\geq 1$, that is, the inequality in \eqref{eq:rank-ineq-strengthened} is not tight. This implies that there exists a tight dicut $\delta^+(U)$ such that the equation $x(\delta^+(U))=1$ is linearly independent of the equations $x(\delta(v)) = 1,~\forall v\in V^t$. In particular, $\1_{\delta^+(U)}$ is linearly independent of the vectors $\1_{\delta(v)},v\in V^t$. Given that
	$$
	 \1_{\delta^+(U)} 
	 = \sum_{u\in \sources(U)} \1_{\delta^+(u)} - \sum_{v\in \sinks(U)} \1_{\delta^-(v)}
	 = \sum_{v\in \sinks(V\setminus U)} \1_{\delta^-(v)} - \sum_{u\in \sources(V\setminus U)} \1_{\delta^+(u)},
	 $$ it follows that both $U,V\setminus U$ must contain active sources, \b{for if not, then one of the equalities above would imply that $\1_{\delta^+(U)}$ is a linear combination of the vectors $\1_{\delta(v)},v\in V^t$, which is a contradiction.} Subsequently, $|U|\neq 1,|V|-1$, so $\delta^+(U)$ is a non-trivial tight dicut, thereby finishing the proof.
\end{proof}

Let us give an intuitive interpretation of the slack $s:=s(D,\zF)$. The affine hull of $F(D,\zF)$ is described by two types of constraints: $x(\delta(v))=1,~\forall v\in V^t$, and $x(\delta^+(U))=1$ for a non-trivial tight dicut $\delta^+(U)$. The slack $s$ computes the additional contribution of non-trivial tight dicuts ---in terms of rank increase--- in defining the affine hull. Furthermore, if $s\geq 1$, then there exists a cross-free family of $s$ non-trivial tight dicuts, which can be used to give a `decomposition' of the digraft into $s+1$ pieces partitioning the active sources of $(D,\zF)$ into $s+1$ nonempty parts, such that each piece of the form $(D',\zF')$ satisfies $s(D',\zF')=\kappa^t(D',\zF')= 0$. Though these ideas can be formalized, we refrain from doing so here as it is outside the scope of this paper.

\subsection{Dicut contractions}\label{subsec:de-composition}

Let $(D=(V,A),\zF)$ be a digraft. In this subsection, we show how to decompose $(D,\zF)$ along certain dicuts into two smaller digrafts. We need a few preliminaries.

\begin{DE}[closure]
The \emph{closure of $\zF$} for $(D,\zF)$ is the family of subsets $U\subset V,U\neq \emptyset$ such that $\delta^-(U)=\emptyset$ and $F(D,\zF)\subseteq \{x:x(\delta^+(U))=1\}$. 
\end{DE}

The following lemma will be useful in this subsection.

\begin{LE}\label{crossing-family}
Let $(D=(V,A),\zF)$ be a digraft, and let $\zC$ be the closure of $\zF$ for $(D,\zF)$. Then $\zF\subseteq \zC$, and $\zC$ is a crossing family.
\end{LE}
\begin{proof}
The inclusion $\zF\subseteq \zC$ is clear. Let $U,W\in \zC$ cross, and let $x\in F(D,\zF)$. As $\delta^+(U),\delta^+(W)$ are dicuts, then so are $\delta^+(U\cap W),\delta^+(U\cup W)$. Subsequently, $x(\delta^+(U\cap W)),x(\delta^+(U\cup W))\geq 1$, and so $$
2 = x(\delta^+(U))+x(\delta^+(W)) = x(\delta^+(U\cap W))+x(\delta^+(U\cup W))\geq 2.
$$ Equality must hold throughout, so $x(\delta^+(U\cap W))=x(\delta^+(U\cup W))= 1$. As this holds for all $x\in F(D,\zF)$, it follows that $U\cap W,U\cup W\in \zC$.
\end{proof}

Of interest are those `non-trivial' dicut inequalities that expose a nonempty face of $F(D,\zF)$.

\begin{DE}[contractible dicut]
	A dicut $\delta^+(U)$ is \emph{contractible} if $1<|U|<|V|-1$ and $F(D,\zF)\cap \{x:x(\delta^+(U))=1\}\neq \emptyset$. 
\end{DE}

The following operation justifies the choice of the terminology above.

\begin{DE}[$(U,V\setminus U)$-contractions]\label{def:dicut-contraction}
	Suppose $\delta^+(U)$ is a contractible dicut of $(D,\zF)$. 
	Let $\overline{\zF}$ be the closure of $\zF\cup \{U\}$ for the digraft $(D,\zF\cup \{U\})$.
	Let $U_1:=U$ and $U_2:=V\setminus U$. Let $D_i=(V_i,A_i)$ be the bipartite digraph obtained from $D$ after shrinking $U_i$ to a single node $u_i$; so $V_i=\{u_i\}\cup U_{3-i}$. Let $$
	\zF_i := \{W:W\cap U_i=\emptyset, W\in \overline{\zF}\}\cup \{(W\setminus U_i)\cup \{u_i\}:U_i\subseteq W, W\in \overline{\zF}\}.
	$$ We refer to $(D_i,\zF_i),i=1,2$ as the \emph{$(U,V\setminus U)$-contractions of $(D,\zF)$}. 
\end{DE}

Note that $u_1$ is a source in $D_1$ and $\{u_1\}\in \zF_1$, and $u_2$ is a sink in $D_2$ and $V_2\setminus u_2=U_1\in \zF_2$. \b{
See \Cref{fig:tight-dicut-decomposition} for an illustration of a $(U,V\setminus U)$-contraction of the digraft on the left for $U:=\{2,3,6,7\}$, for which $\delta^+(U)$ is a contractible dicut. For this instance, $\overline{\zF}$ consists of $\zF$ as well as $\{1\},U$. The two $(U,V\setminus U)$-contractions of this instance are shown on the right. The top digraft $(D_1,\zF_1)$ is balanced basic, where $\zF_1$ consists of $V_1\setminus v$ for all sinks $v$, as well as $\{1\},\{u_1\}$. The bottom digraft $(D_2,\zF_2)$ is skewed basic, where $\zF_2$ consists only of $V_2\setminus v$ for all sinks $v$.
}

\begin{figure}[ht]
	\centering
	\includegraphics[scale=0.4]{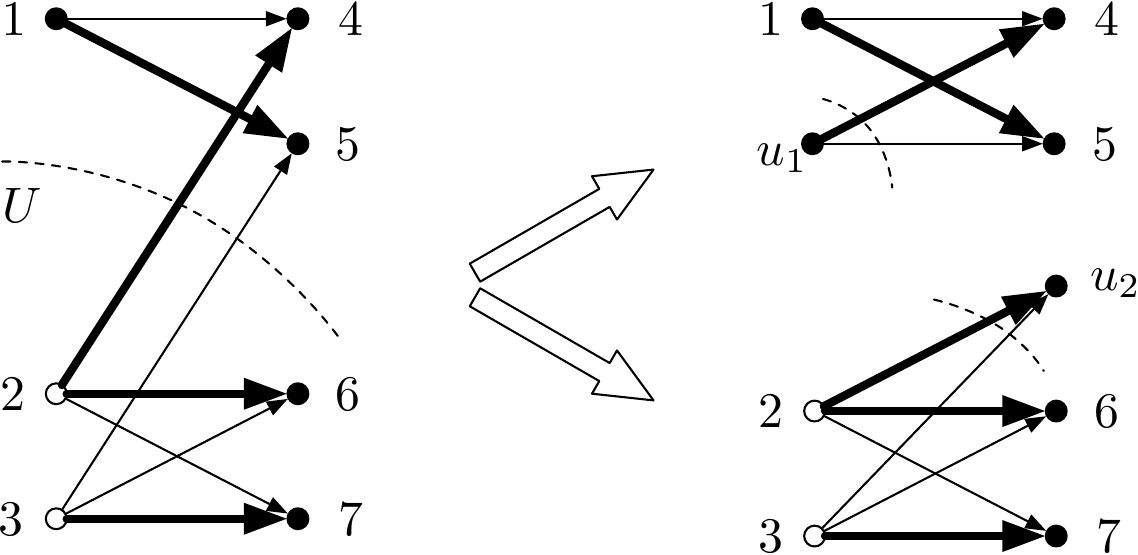}
\caption{\b{(Left) A digraft $(D,\zF)$ where $\zF$ consists of $V\setminus v$ for all sinks $v$. Denote by $J$ the set of bold arcs.
(Right: top, bottom) The two $(U,V\setminus U)$-contractions $(D_i,\zF_i),i=1,2$. Denote by $J_i$ the set of bold arcs of $(D_i,\zF_i)$ for $i=1,2$.
}
}
	\label{fig:tight-dicut-decomposition}
\end{figure}

We now explain how a $(U,V\setminus U)$-contraction decomposes the digraft $(D,\zF)$, as well as the $0,1$ points in $F(D,\zF)\cap \{x:x(\delta^+(U))=1\}$, into digrafts $(D_i,\zF_i)$, and $0,1$ points in $F(D_i,\zF_i)$, respectively. We also explain when and how two $0,1$ points in $F(D_i,\zF_i),i=1,2$ can be composed to give a $0,1$ point in $F(D,\zF)\cap \{x:x(\delta^+(U))=1\}$.

\begin{LE}\label{de-composition}
	Let $(D=(V,A),\zF)$ be a digraft, and suppose $\delta^+(U)$ is a contractible dicut. Let $(D_i=(V_i,A_i),\zF_i),i=1,2$ be the $(U,V\setminus U)$-contractions of $(D,\zF)$. Then the following statements hold: \begin{enumerate}
	        \item {\bf Decomposition:} For $i\in \{1,2\}$, $(D_i,\zF_i)$ is a digraft; furthermore, if $J\subseteq A$ satisfies $\1_J\in F(D,\zF)\cap \{x:x(\delta^+(U))=1\}$, then $J_i:=J\cap A_i$ satisfies $\1_{J_i}\in F(D_i,\zF_i)$.
		\item {\bf Composition:} If $J_i\subseteq A_i$ satisfies $\1_{J_i}\in F(D_i,\zF_i),i=1,2$, and $J_1\cap \delta^+(U)=J_2\cap \delta^+(U)$, then $J:=J_1\cup J_2$ satisfies $\1_J\in F(D,\zF)\cap \{x:x(\delta^+(U))=1\}$. 
	\end{enumerate}
	\b{See \Cref{fig:tight-dicut-decomposition} for an illustration of these operations.}
\end{LE}
\begin{proof}
	{\bf (1)} It can be readily checked that $(D_i,\zF_i)$ is a digraft, for each $i\in \{1,2\}$. Suppose $J\subseteq A$ satisfies $\1_J\in F(D,\zF)\cap \{x:x(\delta^+(U))=1\}$. Observe that $F(D,\zF)\cap \{x:x(\delta^+(U))=1\} = F(D,\zF\cup \{U\})$. Let $\overline{\zF}$ be the closure of $\zF\cup \{U\}$ for $(D,\zF\cup \{U\})$. Then, by definition, $\1_J\in F(D,\overline{\zF})$. It can be readily checked now that $J_i:=J\cap A_i$ satisfies $\1_{J_i}\in F(D_i,\zF_i)$, for each $i\in \{1,2\}$. 
	
	{\bf (2)} We follow the notation in \Cref{def:dicut-contraction}. Suppose $J_1\cap \delta_{D}^+(U)=J_2\cap \delta_D^+(U)=\{a\}$. Then $J\cap \delta_D^+(U)=\{a\}$. Let $x^i:= \1_{J_i}\in F(D_i,\zF_i),i=1,2$, and let $x:=\1_{J}\in \cR^A$. Clearly, $x(\delta^+(U))=1$. 
	Let $C:=\delta^+_D(W)$ be a dicut of $D$. We need to show that $x(C)\geq 1$, and equality holds if $W\in \zF$.
	
	If $W\subseteq U_1$ or $W\supseteq U_2$, then $C$ is also a dicut of $D_2$, and $J\cap C = J_2\cap C\neq \emptyset$, so $x(C)=x^2(C)\geq 1$. Furthermore, if $W\in \zF$ and $W\subseteq U_1$, then $W\in \zF_2$ so $x(C)=x^2(C)=1$, and if $W\in \zF$ and $W\supseteq U_2$, then $(W\setminus U_2)\cup \{u_2\}\in \zF_2$ so $x(C)=x^2(C)=1$. 
	
	Similarly, if $W\subseteq U_2$ or $W\supseteq U_1$, then $C$ is also a dicut of $D_1$, $x(C)=x^1(C)\geq 1$, and if $W\in \zF$, then $x(C)=1$.
	
	Otherwise, $W$ and $U$ cross. Subsequently, 
	$C_1:=\delta_D^+(U\cup W)$ is a dicut of both $D$ and $D_1$, 
	$C_2:=\delta_D^+(U\cap W)$ is a dicut of both $D$ and $D_2$, 
	and therefore the inequality below holds: $$
	x(C) = x(C_1)+x(C_2) - x(\delta^+_D(U)) =
	x^1(C_1)+x^2(C_2)-1\geq 1+1-1=1.
	$$ 
	Suppose $W\in \zF$. It remains to show that $x(C)=1$. Recall that $\overline{\zF}$ is the closure of $\zF\cup \{U\}$ for $(D,\zF\cup \{U\})$. By \Cref{crossing-family}, $\overline{\zF}$ is a crossing family. Thus, given that $W,U\in \zF\cup \{U\}\subseteq \overline{\zF}$ and $U,W$ cross, it follows that $U\cup W,U\cap W\in \overline{\zF}$. Subsequently, by \Cref{def:dicut-contraction}, $((U\cup W)\setminus U_1)\cup \{u_1\}\in \zF_1$ and $U\cap W\in \zF_2$, so $x^1(C_1)=x^2(C_2)=1$, implying in turn that $x(C)=1$, as required.
\end{proof}

We only use \Cref{de-composition}~(decomposition) in this section, as composition will not be needed until the next section.

\subsection{Nontrivial facet-defining inequalities}\label{subsec:FDI}

As we saw in \Cref{de-composition}, a contractible dicut (if any) can be used to decompose the digraft into two smaller ones, and under certain conditions we can also compose two solutions to go back. As every non-trivial tight dicut is clearly contractible, we shall repeatedly contract them to decompose our digraft into basic pieces.
	
	Observe that if $(D,\zF)$ is basic, then every tight dicut is trivial, so $s(D,\zF)=0$ by \Cref{slack=0}.
	
	Given a basic digraft, a key idea that allows for the proof of \Cref{main-digraft} is to contract more dicuts, namely those that correspond to `non-trivial' facet-defining inequalities. What enables these contractions is the following crucial theorem, proving that both pieces of the contraction also have zero slack.

\begin{theorem}\label{slack-cut-contraction}
Let $(D=(V,A),\zF)$ be a basic digraft. Suppose $x(\delta^+(U))\geq 1$ is a facet-defining dicut inequality for $F(D,\zF)$ that is not equivalent to $x(\delta^+(u))\geq 1$ for any active source $u$. Let $U_1:=U$ and $U_2:=V\setminus U$. Let $(D_i=(V_i,A_i),\zF_i),i=1,2$ be the $(U,V\setminus U)$-contractions of $(D,\zF)$. Then the following statements hold for each $i\in \{1,2\}$: \begin{enumerate}
\item $U_i$ contains an active source of $(D,\zF)$,
\item every active source for $(D,\zF)$ in $U_{3-i}$ is an active source for $(D_i,\zF_i)$, and vice versa,
\item $s(D_i,\zF_i) = 0$. 
\end{enumerate}
\end{theorem}
\begin{proof}
Let $U_1:=U$, $U_2:=V\setminus U$, and write $V_i=\{u_i\}\cup U_{3-i}$ for $i=1,2$. 

\begin{claim} 
	For each $i\in \{1,2\}$, $U_i$ contains an active source of $(D,\zF)$, say $z_i$. Thus, part {\bf (1)} holds.
\end{claim}
\begin{cproof}
	Let $x\in F(D,\zF)$. Then we have \begin{align*}
	x(\delta^+(U)) 
	&= \sum_{u\in \sources(U_1)} x(\delta^+(u)) -  \sum_{u\in \sinks(U_1)} x(\delta^-(u)) 
	\\
	& =  \sum_{v\in \sinks(U_2)} x(\delta^-(v)) -  \sum_{v\in \sources(U_2)} x(\delta^+(v)). 
	\end{align*} 
	Suppose for a contradiction that one of $U_1,U_2$ does not contain an active source.
	If $U_1$ consists solely of tight nodes, then $x(\delta^+(U)) = -\disc(U_1)$, and otherwise, $U_2$ consists solely of {tight} nodes, so $x(\delta^+(U)) = \disc(U_2)$. In both cases, we obtain that $x(\delta^+(U))$ has a fixed value for all $x\in F(D,\zF)$, which is a contradiction as $x(\delta^+(U))= 1$ determines a facet, hence proper face of $F(D,\zF)$.~\end{cproof}

Let $V^t,V^a$ be the sets of tight and active nodes of $(D,\zF)$, respectively.

\begin{claim} 
For each $i\in \{1,2\}$, the set of tight nodes for $(D_i,\zF_i)$ is precisely $(V^t\cap U_{3-i})\cup \{u_i\}$. Thus, part {\bf (2)} holds.
\end{claim}
\begin{cproof}
As $(D,\zF)$ is a basic digraft, $\{u\}\in \zF$ for each source $u$ in $V^t$, and $V\setminus u\in \zF$ for each sink $u$ in $V^t$. Thus, by the definition of $(D_i,\zF_i)$, the set of its tight nodes contains $(V^t\cap U_{3-i})\cup \{u_i\}$. To see the reverse inclusion, let $u\in V_i$ be a tight node of $(D_i,\zF_i)$. If $u=u_i$, then clearly $u\in (V^t\cap U_{3-i})\cup \{u_i\}$. Otherwise, $u\in U_{3-i}$.

We first prove that if $x\in F(D,\zF)\cap \{x:x(\delta^+(U))=1\}$ then $x(\delta(u))=1$. It suffices to show this for $0,1$ vectors $x$. To this end, let $J\subseteq A$ such that $\1_J\in F(D,\zF)\cap \{x:x(\delta^+(U))=1\}$. Then by \Cref{de-composition}~(decomposition), $\1_{J_i}\in F(D_i,\zF_i)$ for $J_i=J\cap A_i$, so given that $u$ is a tight node for $(D_i,\zF_i)$, we obtain that $|J_i\cap \delta_{D_i}(u)|=1$, which implies that $|J\cap \delta_{D}(u)|=1$.

Subsequently, the facet of $F(D,\zF)$ obtained by setting $x(\delta^+(U))\geq 1$ to equality satisfies the equation $x(\delta(u))=1$. Given that $x(\delta^+(U))\geq 1$ is a facet-defining inequality for $F(D,\zF)$ that is not equivalent to $x(\delta(v))\geq 1$ for any active source $v$, it follows that $u$ must be a tight node of $(D,\zF)$. Thus, $u\in V^t\cap U_{3-i}$.
\end{cproof} 

It remains to prove part {\bf (3)}, i.e., $s(D_i,\zF_i)=0,i=1,2$. Suppose for a contradiction that $s(D_i,\zF_i)\geq 1$ for some $i\in \{1,2\}$. By \Cref{slack>=1}, $(D_i,\zF_i)$ has a non-trivial tight dicut $\delta_{D_i}(W)$ such that $W$ separates a pair of active sources of $(D_i,\zF_i)$, say $u\in V_i\setminus W$ and $w\in W$, which are also active for $(D,\zF)$ by Claim~2. Note that $u,w\neq u_i$.

We may assume that $u_i\notin W$ by replacing $W$ with $V_i\setminus W$, if necessary. Let us now consider the dicut $\delta_{D}(W)$ in $D$. (Note that $\delta_{D}(W)$ is equal to one of $\delta_D^{\pm}(W)$.) Note that $W\cap U_i=\emptyset$. Furthermore, each of $W,U_i,V\setminus (W\cup U_i)$ contains an active source of $(D,\zF)$, namely $w\in W, z_i\in U_i$, and $u\in V\setminus (W\cup U_i)$, where $z_i$ comes from Claim~1.

\begin{claim} 
	$F(D,\zF)\cap \{x:x(\delta_D(U_i))=1\}\subseteq \{x:x(\delta_D(W))=1\}$.
\end{claim}
\begin{cproof}
It suffices to prove this inclusion for $0,1$ vectors $x$. To this end, take $J\subseteq A$ such that $\1_J\in F(D,\zF)\cap \{x:x(\delta^+_D(U))=1\}$. Let $J_i:=J\cap A_i$. Then $\1_{J_i}\in F(D_i,\zF_i)$ by \Cref{de-composition}~(decomposition). Thus, $|J_i\cap \delta_{D_i}(W)|=1$ as $\delta_{D_i}(W)$ is a tight dicut for $(D_i,\zF_i)$, implying in turn that $|J\cap \delta_D(W)|=1$, as required. 
\end{cproof}
	
It follows from Claim~3 that $\delta_D(W)$ is either a tight dicut for $(D,\zF)$, or $x(\delta_D(W))\geq 1$ is a dicut inequality that defines the same facet of $F(D,\zF)$ as $x(\delta_D(U_i))\geq 1$. The former is not possible as $(D,\zF)$ is a basic digraft and $|W|\neq 1,|V|-1$. Subsequently, $x(\delta_D(W))=1$ and $x(\delta_D(U_i))=1$ must define the same facet. This implies that the first equation must be implied by the second equation, together with all the equations that define the affine hull of $F(D,\zF)$, which are of the form $x(\delta(v))=1,\,v\in V^t$ because $(D,\zF)$ is basic. 

Subsequently, we must have that $\1_{\delta(W)}$ is in the linear hull of $\1_{\delta(U_i)}$ and $\1_{\delta(v)},v\in V^t$. Subsequently, the vectors $\1_{\delta(W)},\1_{\delta(U_i)}$ and $\1_{\delta(v)},v\in V^t$ are linearly dependent. We have \begin{align*}
\pm\1_{\delta(W)} &= \sum_{v\in \sources(W)} \1_{\delta(v)} - \sum_{v\in \sinks(W)} \1_{\delta(v)}\\
\pm\1_{\delta(U_i)} &= \sum_{v\in \sources(U_i)} \1_{\delta(v)} - \sum_{v\in \sinks(U_i)} \1_{\delta(v)}.
\end{align*} Thus, after applying elementary `row' operations, the following vectors are linearly dependent: \begin{align*}
\sum_{v\in \sources(W)\cap V^a} \1_{\delta(v)}&\\
\sum_{v\in \sources(U_i)\cap V^a} \1_{\delta(v)}&\\
\1_{\delta(v)}& \quad \forall v\in V^t.
\end{align*} Note that $\sources(W)\cap V^a\ni w$ and $\sources(U_i)\cap V^a\ni z_i$ are nonempty and disjoint. Thus, given that $u\in (V\setminus W\setminus U_i)\cap V^a$, the linear dependence of the vectors above implies that of the vectors $\1_{\delta(v)},v\in V\setminus u$, which is a contradiction to \Cref{linear-independence-RE}. This finishes the proof of (3).
\end{proof}

\subsection{Proof of the Affine Critical Lemma}\label{subsec:affine-critical-LE}

Recall that by definition, a digraft $(D,\zF)$ is affine critical if $\aff(F(D,\zF))=\big\{x:x(\delta(v))=1,\,\forall v\in V^t\big\}$. As an immediate consequence of \Cref{slack=0}, we obtain the following remark.

\begin{RE}\label{slack-0-affine-critical}
A digraft is affine critical if and only if it has slack zero.
\end{RE}

While every basic digraft is affine critical, the converse may not hold, as having slack zero is not sufficient for being basic. We are now ready to prove the Affine Critical Lemma.

\begin{proof}[Proof of \Cref{affine-critical-LE}]
Let $(D=(V,A),\zF)$ be a basic digraft that is not robust, that is, there is a facet-defining dicut inequality $x(\delta^+(U))\geq 1$ for $F(D,\zF)$ that is not equivalent to $x(\delta^+(u))\geq 1$ for any active source~$u$. Clearly, $s(D,\zF)=0$. By \Cref{slack-cut-contraction}, each of $(D,\zF),(D_i,\zF_i),i=1,2$ contains an active source, for $i\in \{1,2\}$ every active source for $(D,\zF)$ in $U_{3-i}$ is an active source for $(D_i,\zF_i)$ and vice versa, and $s(D_i,\zF_i) = 0,i=1,2$, so by \Cref{slack-0-affine-critical}, $(D_i,\zF_i),i=1,2$ are affine critical digrafts, thus finishing the proof.
\end{proof}

\section{Proof of \Cref{main-digraft}}\label{sec:main-proof}

The proof proceeds by induction \b{on the number of vertices of the digraft}, with the base case being \b{that of} basic robust digrafts.

\begin{proof}[Proof of \Cref{basic-robust-IDP}]
Let $(D=(V,A),\zF)$ be a basic robust digraft. We need to show that $F(D,\zF)$ has the integer decomposition property, and $\aff(F(D,\zF)) = \{x:Mx=\1\}$ for some $M\in \cZ^{m\times n}$ with $m\geq 1$. To this end, let $V^t,V^a$ be the sets of tight and active nodes of $(D,\zF)$, respectively. Let $P:=F(D,\zF)$. It follows from the hypothesis that every facet-defining inequality of $P$ is either equivalent to $x_a\geq 0$ for some $a\in A$, or $x(\delta^+(v))\geq 1$ for some $v\in V^a$. Furthermore, by \Cref{arc-active-CO}, $\aff(P)$ is described by $x(\delta(v))=1,\,\forall v\in V^t$; this proves the second part of the theorem. For the first part, let us write
	 $$P=
	\left\{x\in \cR^A:x\geq \0;\,x(\delta^+(v))\geq 1,\,\forall v\in V^a;\, x(\delta(v))=1,\,\forall v\in V^t\right\}.
	$$ Take an integer $k\geq 1$, and let $x\in \cZ_{\geq 0}^A$ be an integral vector in $kP$. That is, $x\geq \0$, $x(\delta(v))=k$ for all $v\in V^t$, and $x(\delta^+(v))\geq k$ for all $v\in V^a$. By a result of de Werra (\cite{deWerra71}, see \cite{LP09}, Corollary 1.4.21) on edge-colourings of bipartite graphs, we can write $x$ as the sum of $x^1,\ldots,x^k\in \cZ_{\geq 0}^A$ such that for each $i$, $x^i(\delta(v))=1$ for all $v\in V^t$ and $x^i(\delta^+(v))= \lfloor x(\delta^+(v))/k\rfloor$ or $\lceil x(\delta^+(v))/k\rceil$ for all $v\in V^a$. In particular, each $x^i$ belongs to $P$. This finishes the proof.
\end{proof}

For the induction step, we will need to compose integral bases along contractible dicuts. The following technical though straightforward lemma will be useful for this purpose. Given subsets $A_1,A_2\subseteq A$ such that $A_1\cap A_2=C$ and $A_1\cup A_2=A$, and weights $w^i\in \cR^{A_i}$ such that $w^1_a=w^2_a,\,\forall a\in C$, we define $z:=w^1\odot w^2\in \cR^A$ as follows: $z_a:=w^i_a$ if $a\in A_i\setminus A_{3-i}$, and $z_a:=w^1_a=w^2_a$ if $a\in C$.

\begin{LE}\label{basis-going-up}
Let $(D=(V,A),\zF)$ be a digraft, and suppose $\delta^+(U)$ is a contractible dicut. Let $(D_i=(V_i,A_i),\zF_i),i=1,2$ be the $(U,V\setminus U)$-contractions of $(D,\zF)$. Suppose 
\begin{align*}
B_1&:=\{x^1,\ldots,x^{d_1}\}\subseteq F(D_1,\zF_1)\cap \{0,1\}^{A_1}\\
B_2&:=\{y^1,\ldots,y^{d_2}\}\subseteq F(D_2,\zF_2)\cap \{0,1\}^{A_2}
\end{align*}
are integral bases for $\lin(F(D_1,\zF_1))$ and $\lin(F(D_2,\zF_2))$, respectively. For each $a\in \delta^+(U)$, let $I_a:=\{i:x^i_a=1\}$ and $J_a:=\{j:y^j_a=1\}$. Then both $I_a,J_a$ are nonempty. Furthermore, write $I_a=\{i_1,\ldots,i_k\}$ and $J_a=\{j_1,\ldots,j_\ell\}$, and let \begin{align*}
			z^a_t&:=x^{i_1} \odot y^{j_t} \quad &t=1,\ldots,\ell\\
			z^a_{\ell+t}&:=x^{i_{1+t}}\odot y^{j_1}  &t=1,\ldots,k-1.
		\end{align*} 
	Let $B_1\odot B_2:=\{z^a_i:a\in \delta^+(U), 1\leq i\leq |I_a|+|J_a|-1\}$. 
	Then the following statements hold: \begin{enumerate}
	\item $B_1\odot B_2\subseteq F(D,\zF)\cap \{0,1\}^A\cap \{x:x(\delta^+(U))=1\}$, $|B_1\odot B_2| = d_1+d_2-|\delta^+(U)|$, and $B_1\odot B_2$ is linearly independent,
	\item if $x$ is an integer linear combination of the vectors in $B_1$ and $y$ of $B_2$, where $x_a=y_a\, \forall a\in \delta^+(U)$, then $x\odot y$ is an integer linear combination of the vectors in~$B_1\odot B_2$,
	\item $B_1\odot B_2$ is an integral basis for its linear hull,
	\item if $\delta^+(U)$ is a tight dicut, then $\lin(B_1\odot B_2) = \lin(F(D,\zF))$. 
	\end{enumerate}
\end{LE}
\begin{proof} 
Given that $\delta^+(U)$ is a contractible dicut, it follows that $(D_i,\zF_i),i=1,2$ are digrafts, so by \Cref{arc-active-CO}, both $I_a,J_a$ are nonempty. Let $B:=B_1\odot B_2$. 

{\bf (1)} 
It follows from \Cref{de-composition}~(composition) that $B\subseteq F(D,\zF)\cap \{0,1\}^A$. Furthermore, it is clear from construction that
$x(\delta^+(U))=1$ for all $x\in B$, and {$$
|B| = \sum_{a\in \delta^+(U)} (|I_a|+|J_a|-1)=\sum_{a\in \delta^+(U)} |I_a| + \sum_{a\in \delta^+(U)} |J_a| - |\delta^+(U)| = |B_1|+|B_2| - |\delta^+(U)|;
$$ the last identity follows from the fact that $(I_a:a\in \delta^+(U)),(J_a:a\in \delta^+(U))$ partition $B_1,B_2$, respectively, as $x(\delta^+(U))=y(\delta^+(U))=1$ for all $x\in B_1,y\in B_2$.
}
To prove linear independence, suppose $\sum_{a,i} \lambda_i^a z_i^a = 0$ for some $\lambda_i^a\in \cR$ for all $a\in \delta^+(U),1\leq i\leq |I_a|+|J_a|-1$.
Fix $a\in \delta^+(U)$ with $I_a=\{i_1,\ldots,i_k\}$ and $J_a=\{j_1,\ldots,j_\ell\}$. 
Given that $B_1$ is linearly independent, for each $x^{i_t}$, the sum of the coefficients of vectors in $B$ of the form $x^{i_t}\odot y$ for some $y$, must be $0$. Subsequently, we have \begin{align}
\sum_{i=1}^\ell \lambda_{i}^a &=0 \label{eq:lin-ind-1}\\
\lambda^a_{\ell+1} = \cdots =\lambda^a_{\ell+k-1}&=0\label{eq:lin-ind-2}
\end{align} where \eqref{eq:lin-ind-1} computes the coefficient \b{sum} for $x^{i_1}\odot y$, while \eqref{eq:lin-ind-2} computes the {coefficient sums} for $x^{i_t}\odot y, t=2,\ldots,k$. Similarly, given that $B_2$ is linearly independent, for each $y^{j_t}$, the sum of the coefficients of vectors in $B$ of the form $x\odot y^{j_t}$ for some $x$, must be $0$. Subsequently, we obtain that 
\begin{align}
\b{\lambda_1^a +} \sum_{i=\ell \b{+1}}^{\ell+k-1} \lambda_{i}^a &= 0\label{eq:lin-ind-3}\\
\lambda^a_{2} = \cdots =\lambda^a_{\ell}&=0\label{eq:lin-ind-4}
\end{align} where \eqref{eq:lin-ind-3} computes the coefficient \b{sum} for $x\odot y^{j_1}$, while \eqref{eq:lin-ind-4} computes the {coefficient sums} for $x\odot y^{j_t}, t=2,\ldots,\ell$. Observe that \eqref{eq:lin-ind-1} and \eqref{eq:lin-ind-4} imply that $\lambda_1^a=0$, so together with \eqref{eq:lin-ind-2}, we obtain that $\lambda^a_i=0$ for all $1\leq i\leq k+\ell-1$. As this holds for all $a\in \delta^+(U)$, we obtain that $\lambda^a_i=0$ for all $a\in \delta^+(U),1\leq i\leq |I_a|+|J_a|-1$. 

{\bf (2)}
Suppose $x=\sum_{i} \alpha(x^i) x^i$ and $y=\sum_{j} \beta(y^j) y^j$ for integers $\alpha(x^i)$ and $\beta(y^j)$. Fix $a\in \delta^+(U)$ with $I_a=\{i_1,\ldots,i_k\}$ and $J_a=\{j_1,\ldots,j_\ell\}$. Now choose $\lambda^a_i$ for all $1\leq i\leq k+\ell-1$ such that 
\begin{align}
\sum_{i=1}^\ell \lambda_{i}^a &= \alpha(x^{i_1})\label{eq:lin-ind-5}\\
\lambda^a_{\ell+t-1} &= \alpha(x^{i_t}) \quad t=2,\ldots,k\label{eq:lin-ind-6}\\
\lambda^a_{t} &= \beta(y^{j_t}) \quad t=2,\ldots,\ell.\label{eq:lin-ind-7}
\end{align} \eqref{eq:lin-ind-6} and \eqref{eq:lin-ind-7} give us the values for $\lambda^a_t,t=2,\ldots,\ell+k-1$, all of which are clearly integral. Furthermore, \eqref{eq:lin-ind-5} and \eqref{eq:lin-ind-7} give us an integer value for $\lambda^a_1$: $$
\lambda^a_1 = \alpha(x^{i_1}) - \sum_{t=2}^{\ell} \beta(y^{j_t}).
$$ Since $x_a=y_a$, {$x_a=\sum_{t=1}^{k} \alpha(x^{i_t})$, and $y_a = \sum_{t=1}^{\ell} \beta(y^{j_t})$, it follows} that 
\b{$\beta(y^{j_1}) - \sum_{t=2}^{k} \alpha(x^{i_t})  = 
\alpha(x^{i_1}) - \sum_{t=2}^{\ell} \beta(y^{j_t}) = \lambda^a_1$}, so {
\begin{align}
\lambda_1^a+\sum_{i=\ell+1}^{\ell+k-1} \lambda_{i}^a &= \beta(y^{j_1}).\label{eq:lin-ind-8}
\end{align}}
\b{
It follows from \eqref{eq:lin-ind-5}-\eqref{eq:lin-ind-8} that \begin{align*}
\sum_{i=1}^{k+\ell-1} \lambda^a_i z^a_i 
&= 
\lambda_1^a (x^{i_1} \odot y^{j_1}) + 
\sum_{t=2}^{\ell} \lambda^a_t (x^{i_1}\odot y^{j_t}) + 
\sum_{t=2}^{k} \lambda^a_{\ell+t-1} (x^{i_{t}}\odot y^{j_1})\\
&=
\left(\sum_{t=1}^k \alpha(x^{i_t}) x^{i_t}\right)\odot \left(\sum_{t=1}^{\ell} \beta(y^{j_t}) y^{j_t}\right)
\end{align*} As this holds for all $a\in \delta^+(U)$, it follows that $\sum_{a,i} \lambda^a_i z^a_i = x\odot y$, so $x\odot y$ is an integer linear combination of the vectors in $B$, as required.
}

{\bf (3)} Let $f\in \lin(B)\cap \cZ^A$. Observe that $f=x\odot y$, where $x\in \lin(B_1)\cap \cZ^{A_1}$ and $y\in \lin(B_2)\cap \cZ^{A_2}$. As $B_1$ (resp.\ $B_2$) is an integral basis, $x$ (resp.\ $y$) must be an integer linear combination of the vectors in the set, so by part (2), $f = x\odot y$ is an integer linear combination of the vectors in $B$.

{\bf (4)} Clearly, $\lin(B)\subseteq \lin(F(D,\zF))$. For the reverse inclusion, pick a $0,1$ vector $f\in F(D,\zF)$. Take $J\subseteq A$ such that $\1_J=f$, and let $J_i:=J\cap A_i,i=1,2$. As $\delta^+(U)$ is a tight dicut, we have $f(\delta^+(U))=1$, so we get from \Cref{de-composition}~(decomposition) that $f_i:=\1_{J_i}\in F(D_i,\zF_i),i=1,2$. Subsequently, $f_i\in \lin(B_i),i=1,2$, so $f=f_1\odot f_2\in \lin(B)$.
\end{proof}

We are now ready to prove \Cref{main-digraft}.

\begin{proof}[Proof of \Cref{main-digraft}]
Let $(D=(V,A),\zF)$ be a digraft. Our goal is to prove that $F(D,\zF)\cap \{0,1\}^A$ contains an integral basis for $\lin(F(D,\zF))$. To this end, we may assume that whenever $\delta^+(U)$ is a tight dicut, then $U$ belongs to $\zF$, by adding it to the family if necessary. Note that this operation does not change the face $F(D,\zF)$.

\paragraph{Base case.} We shall proceed by induction \b{on the number of vertices of the digraft}. If $(D,\zF)$ is a basic robust digraft, then $F(D,\zF)$ has the integer decomposition property, and $\aff(F(D,\zF))=\{x:Mx=\1\}$ for some $M\in \cZ^{m\times n}$ with $m\geq 1$, by \Cref{basic-robust-IDP}. It therefore follows from \Cref{IDP-integral-basis} that $F(D,\zF)\cap \{0,1\}^A$ contains an integral basis $B$ for $\lin(F(D,\zF))$, so we have proved the base case. For the induction step, let $(D,\zF)$ be a digraft that is either non-basic or basic non-robust. 

\paragraph{Non-basic case.} Assume in the first case that $(D,\zF)$ is not basic. Thus, by the maximality of $\zF$, there exists a tight dicut $\delta^+(U)$ such that $1<|U|<|V|-1$. Let $D_i=(V_i,A_i),i=1,2$ be the $(U,V\setminus U)$-contractions of $(D,\zF)$. By the induction hypothesis, $F(D_i,\zF_i)\cap \{0,1\}^{A_i}$ contains an integral basis $B_i$ for $\lin(F(D_i,\zF_i))$, for $i\in \{1,2\}$. Then by \Cref{basis-going-up} parts (1), (3) and (4), $B_1\odot B_2\subseteq F(D,\zF)\cap \{0,1\}^A$ is an integral basis for $\lin(F(D,\zF))$.

\paragraph{Basic non-robust case.} Assume in the remaining case that $(D,\zF)$ is a basic non-robust digraft. Then there exists a facet-defining dicut inequality $x(\delta^+(U))\geq 1$ for $F(D,\zF)$ that is not equivalent to $x(\delta^+(u))\geq 1$ for any active source $u$. In particular, $|V|-1>|U|>1$. Let $U_1:=U, U_2:=V\setminus U$, and let $D_i=(V_i,A_i),i=1,2$ be the $(U,V\setminus U)$-contractions of $(D,\zF)$. 

By the induction hypothesis, $F(D_i,\zF_i)\cap \{0,1\}^{A_i}$ contains an integral basis $B_i$ for $\lin(F(D_i,\zF_i))$, for $i\in \{1,2\}$. Then by \Cref{basis-going-up} parts (1) and (3), $B':=B_1\odot B_2\subseteq F(D,\zF)\cap \{0,1\}^A$ is an integral basis for its linear hull. Unlike the previous case, $\lin(B')$ is no longer the same as $\lin(F(D,\zF))$, and we must add at least one element to $B'$.

By the Jump-Free Lemma (i.e., \Cref{jump-free-LE}), there exists $b\in F(D,\zF)\cap \{0,1\}^A$ such that $b(\delta^+(U))=2$. We claim that $B:=B'\cup \{b\}\subseteq F(D,\zF)\cap \{0,1\}^A$ is an integral basis for $\lin(F(D,\zF))$.

First, we prove that $B$ is a linear basis for $\lin(F(D,\zF))$. 
{It follows from \Cref{basis-going-up} part (1) that $B'$ is linearly independent. Furthermore, $b$ is not a linear combination of $B'$. For if it were, then $b = \sum_{z\in B'} \alpha_z z$ for some $\alpha_z\in \cR, z\in B'$. In particular, for any tight node $v$, $$1=b(\delta(v)) = \sum_{z\in B'} \alpha_z z(\delta(v)) = \1^\top z.$$
On the other hand,
$$
2 = b(\delta^+(U)) = \sum_{z\in B'} \alpha_z z(\delta^+(U)) = \1^\top z,$$ where we used $z(\delta^+(U))=1$ for all $z\in B'$ as guaranteed by \Cref{basis-going-up} part (1). However, $\1^\top z$ cannot be both $1$ and $2$, a contradiction. Thus, $B$ is linearly independent.} To show that $B$ is a linear basis, we count the linear dimension of $F:=F(D,\zF)$, which is $d:=1+\dim(F)$. We claim that $d=|B|$. To this end, let $F_i:=F(D_i,\zF_i),i=1,2$ and $d_i:=1+\dim(F_i),i=1,2$. By the Affine Critical Lemma (i.e., \Cref{affine-critical-LE}), $(D,\zF),(D_i,\zF_i),i=1,2$ are affine critical digrafts each of which contains an active source. Thus, $\kappa^t(D,\zF)=\kappa^t(D_1,\zF_1)=\kappa^t(D_2,\zF_2)=0$, and $s(D,\zF)=s(D_1,\zF_1)=s(D_2,\zF_2)=0$ by \Cref{slack-0-affine-critical}. Subsequently, by the slack formula in \Cref{def:slack}, 
	\begin{align*}
		d &= 1+|A|-|V^t|\\
		d_1&=1+|A_1|-|V^t_1|\\
		d_2&= 1+|A_2|-|V^t_2|,
	\end{align*} 
	where $V^t,V^t_i,i=1,2$ denote the sets of tight nodes of $(D,\zF),(D_i,\zF_i),i=1,2$, respectively. By \Cref{slack-cut-contraction} part (2), $V^t_1\cup V^t_2 = V^t \cup \{u_1,u_2\}$, implying in turn the first equality below: \begin{align*}
	d 
	&= d_1+d_2 - |\delta^+(U)| +1\\
	&= |B_1| + |B_2| - |\delta^+(U)| +1\\
	&=|B'|+1\\
	&=|B|.
	\end{align*} The second and last equalities are clear, while the third equality follows from \Cref{basis-going-up} part (1).
	
It remains to prove that $B$ is an \emph{integral} basis. To this end, pick an integral vector $f$ in $\lin(F(D,\zF))$. We now know that $f$ can be expressed as a linear combination of the vectors in $B$; let $\lambda_z\in \cR$ be the coefficient of $z\in B$. Given that $f$ is integral, $f(\delta^+(U))$ is an integer, which can be calculated alternatively as follows: 
$$
	f(\delta^+(U)) = \sum_{z\in B} \lambda_z z(\delta^+(U)) = 2\lambda_{b} + \sum_{z\in B'} \lambda_z = \lambda_{b}  + \1^\top \lambda.
	$$ Here we have used $z(\delta^+(U))=1$ for all $z\in B'$, guaranteed by \Cref{basis-going-up} part (1). As $B\subseteq F(D,\zF)$, we have $z(\delta(v))=1, \forall z\in B$ for any fixed tight node $v$. Subsequently, $\1^\top \lambda = \sum_{z\in B} \lambda_z = f(\delta(v))$ is an integer, implying in turn that $\lambda_{b} = f(\delta^+(U))-\1^\top \lambda$ is an integer. Now let $f':=f-\lambda_{b} b\in \cZ^A$. Evidently, $f'\in \lin(B')\cap \cZ^A$, so given that $B'$ is an integral basis for its linear hull, $f'$ is an integer linear combination of the vectors in $B'$, implying in turn that $\lambda$ is an integral vector. Thus, $B\subseteq F(D,\zF)\cap \{0,1\}^A$ is an integral basis for $\lin(F(D,\zF))$, thereby completing the induction step.
\end{proof}

\section{Proofs of \Cref{scr-theorem} and the applications}\label{sec:apps}

In this section, we start off by a useful mapping of the strengthening sets of a digraph to dijoins of another digraph. We then prove a general lattice theoretic fact about faces of $\scr(D)$ for a digraph $D$. After that, we prove \Cref{scr-theorem} and its three applications, namely, \Cref{ARF-partition-CO}, \Cref{p-adic-CO}, and finally \Cref{head-disjoint-CO}. Finally, we provide an important example showing that many of our results are best possible.

\subsection{A useful mapping}

The following theorem will be particularly useful for two of the proofs. The construction given in the proof has appeared before in the literature, e.g., \cite{Schrijver-note,Cornuejols24}.

\begin{theorem}\label{scr->dij}
Let $D=(V,A)$ be a digraph whose underlying undirected graph is $2$-edge-connected. Let $\zF$ be a family over ground set $V$ such that $\emptyset,V\notin \zF$, and the following face of $\scr(D)$ is nonempty: 
$$F:=\scr(D)\cap \left\{x\in \cR^A: x(\delta^+(U))-x(\delta^-(U))=1-|\delta^-(U)|,\, \forall U\in \zF\right\}.$$ Then there exists a digraft $(D',\zF')$ such that the mapping $x\mapsto \big(\begin{smallmatrix} x\\ \1-x\end{smallmatrix}\big)$ defines a bijection between the face $F$ of $\scr(D)$ and the face $F(D',\zF')$ of $\dij(D')$.
\end{theorem}
\begin{proof}
Let $D'=(V',A')$ be the digraph obtained from $D$ by replacing every arc $a:=(r,s)\in A$, by the two arcs $(r,t_a),(s,t_a)$, where $t_a$ is a new node. \b{(See \Cref{fig:STR-to-dijoin} for an illustration of this mapping.)} Subsequently, $V' = V\cup \{t_a:a\in A\}$ and $A' = \{(r,t_a),(s,t_a):a=(r,s)\in A\}$. Note that the nodes of $D'$ in $V$ are sources, while the new nodes in $\{t_a:a\in A\}$ are sinks. Furthermore, the underlying undirected graph of $D'$ is $2$-edge-connected, as this is so for $D$.

Given a nonempty proper subset $U$ of $V$, denote by $$\varphi(U):=U\cup \{t_a:a=(r,s)\in A;\, r,s\in U\}.$$ Note that $\delta^+_{D'}(\varphi(U)) = \{(r,t_a):a=(r,s)\in \delta^+_D(U)\}\cup \{(s,t_a):a=(r,s)\in \delta^-_D(U)\}$ and $\delta^-_{D'}(\varphi(U)) = \emptyset$. Thus, $\varphi$ maps every nonempty proper subset of $V$ to a dicut of $D'$. Conversely, it can be readily checked for $U'\subset V'$ that, if $\delta^+_{D'}(U')$ is a minimal dicut of $D'$ such that $|U'|<|V'|-1$, then $U:=U'\cap V$ is a nonempty proper subset of $V$, and $\delta^+_{D'}(U') = \delta^+_{D'}(\varphi(U))$. 

Given $J\subseteq A$, denote by $$\phi(J):=\{(r,t_a):a=(r,s)\in J\}\cup \{(s,t_a):a=(r,s)\in A\setminus J\}.$$ Using the mapping $\varphi$ defined earlier, it can be readily checked that $J$ is a strengthening set in $D$ if, and only if, $J':=\phi(J)$ is a 
dijoin of $D'$ such that $|J'\cap \delta(t_a)|=1,\, \forall a\in A$. Subsequently, $\phi$ is a bijection between the strengthening sets $J$ in $D$ and the dijoins $J'$ in $D'$ such that $|J'\cap \delta(t_a)|=1,\, \forall a\in A$.

Let $$\zF':=\{\varphi(U):U\in \zF\}\cup \{V'\setminus t_a:a\in A\},$$ and let $F':=F(D',\zF')$. For $x\in \cR^A$, define $x'\in \cR^{V'}$ as follows: for $a=(r,s)\in A$, let $x'_{(r,t_a)}=x_a$ and $x'_{(s,t_a)} = 1-x_a$. Then \begin{align*}
x'(\delta^+_{D'}(\varphi(U))) &= x(\delta^+_D(U)) + |\delta^-_D(U)| - x(\delta^-_D(U)) &\forall U\subset V, U\neq \emptyset\\
x'(\delta^-_{D'}(t_a)) &= x_a+(1-x_a)=1 &\forall a\in A.
\end{align*} Subsequently, if $x\in \scr(D)$, then $x'\in \dij(D')$. Furthermore, $x\in F$ if, and only if, $x'\in F'$. In particular, $F'\neq \emptyset$, and so $(D',\zF')$ is the desired digraft.
\end{proof}

\subsection{Integer lattices and faces of the strongly connected re-orientations polytope}\label{subsec:scr-theorem-proof}

Let us recall some basic concepts of the theory of integer lattices; for a reference textbook we recommend (\cite{Martinet03}, Chapter 1). A subset $L\subseteq \cR^A$ is a \emph{lattice} if it is the set of integer linear combinations of finitely many vectors. Alternatively, $L$ is a lattice if it forms a subgroup of $\cR^A$ under addition that is \emph{discrete}, that is, there exists an $\varepsilon>0$ such that every pair of distinct vectors in $L$ are at distance $\geq \varepsilon$. 
Given a finite subset $G\subset \cR^A$, the \emph{lattice generated by $G$}, denoted $\lat(G)$, is the set of all integer linear combinations of the vectors in $G$. A \emph{lattice basis for $L$} is a set $B$ of linearly independent vectors that generates the lattice, i.e., $L = \lat(B)$. A nontrivial fact is that a lattice basis always exists.

Suppose now $L$ is an \emph{integer lattice}, that is, $L$ is a lattice and $L\subseteq \cZ^A$. Let $\overline{L}:=\lin(L)\cap \cZ^A$ which is another integer lattice that contains $L$. Note that $\overline{L}$ is the `densest' integer lattice in $\lin(L)$. It is known that $\overline{L}$ can be partitioned into a finite number of lattices, each of which is an integral shift of $L$, i.e., of the form $L+w:=\{v+w:v\in L\}$ for some $w\in \lin(L)\cap \cZ^A$. We refer to the number of parts in this partition as the \emph{index of $L$} and denote it by $\ind(L)\in \cZ_{\geq 1}$. Thus, the smaller the index of $L$, the denser the lattice is. Of particular interest is the case when $L$ is densest possible. Observe that $L$ has index $1$ if, and only if, $L$ contains an integral basis for $\lin(L)$.

\begin{theorem}\label{scr-theorem-extension}
	Let $D=(V,A)$ be a digraph whose underlying undirected graph is $2$-edge-connected. Let $\zF$ be a family over ground set $V$ such that $\emptyset,V\notin \zF$, $1-|\delta^-(U)|\neq 0$ for some $U\in \zF$, and the following face of $\scr(D)$ is nonempty:
	$$F:=\scr(D)\cap \left\{x\in \cR^A: x(\delta^+(U))-x(\delta^-(U))=1-|\delta^-(U)|,\, \forall U\in \zF\right\}.$$
	Then the following statements hold: \begin{enumerate}
\item The lattice generated by $F\cap \{0,1\}^A$ has a lattice basis contained in $F\cap \{0,1\}^A$.
\item Let $g:=\gcd\{1-|\delta^-(U)|:U\in \zF\}$. Then $gx\in \lat\left(F\cap \{0,1\}^A\right)$ for all $x\in \lin(F)\cap \cZ^A$.
\end{enumerate} 
	\end{theorem}
\begin{proof}
	 Let $L$ be the lattice generated by $F\cap \{0,1\}^A$, $\overline{L}:=\lin(F)\cap \cZ^A$, and $g:=\gcd\{1-|\delta^-(U)|:U\in \zF\}$. By \Cref{scr->dij}, there exists a digraft $(D'=(V',A'),\zF')$ such that for $F':=F(D',\zF')$, the mapping $f:F\to F'$ defined as $f(x)=\big(\begin{smallmatrix} x\\ \1-x\end{smallmatrix}\big)$ is a bijection. Let $L'$ be the lattice generated by $F(D',\zF')\cap \{0,1\}^{A'}$. By \Cref{main-digraft}, there is an integral basis $B'\subseteq F'\cap \{0,1\}^{A'}$ for $\lin(F')$.

\begin{claim*} 
Let $w$ be an integral vector in $\lin(F)$, expressed as $w = \sum_{x\in F} \lambda_x x$. Then $\1^\top \lambda$ is $\frac{1}{g}$-integral. Furthermore, if $w=\0$, then $\1^\top \lambda=0$.
\end{claim*}
\begin{cproof}
Let $\tau:=\sum_{x\in F} \lambda_x$. Note that $$
w(\delta^+(U))-w(\delta^-(U))= \sum_{x\in F} \lambda_x (1-|\delta^-(U)|) = \tau(1-|\delta^-(U)|) \quad \forall U\in \zF.
$$ As $w$ is integral, we have $\tau(1-|\delta^-(U)|)\in \cZ$ for all $U\in \zF$, and so since $\gcd\{1-|\delta^-(U)|:U\in \zF\}=g$, it follows that $\tau$ is $\frac{1}{g}$-integral. Furthermore, if $w=\0$, then as $1-|\delta^-(U)|\neq 0$ for some $U\in \zF$, we have $0=w(\delta^+(U))-w(\delta^-(U))=\tau(1-|\delta^-(U)|)$, implying in turn that $\tau=0$.
\end{cproof}

Let $B$ be the pre-image of $B'$ under $f$, \b{i.e., $B=\{b\in F:f(b)\in B'\}$}. Observe that $B\subseteq F\cap \{0,1\}^A$. We shall prove that (a) $B$ is linearly independent, (b) $B$ is a lattice basis for $L$, and (c)  $gw\in L$ for all $w\in \lin(F)\cap \cZ^A$. 
\begin{enumerate}
\item[(a)] Suppose $\sum_{b\in B} \lambda_b b =\0$. It follows from the claim above that $\1^\top \lambda=0$, so $\sum_{b\in B} \lambda_b f(b) = \0$. The linear independence of $B'=\{f(b):b\in B\}$, along with the bijectivity of $f$, implies that $\lambda=\0$.

\item[(b)] By (a), it suffices to show that $\lat(B)=L$. Clearly, $\lat(B)\subseteq L$. For the reverse inclusion, let $w\in L$. Then $w = \sum_{x\in F\cap \{0,1\}^A} \lambda_x x$ for some integers $\lambda_x,x\in F\cap \{0,1\}^A$. Let $w':=\sum_{x\in F\cap \{0,1\}^A} \lambda_x f(x)$. As $\lambda$ is integral, $w'\in L'$, so $w' = \sum_{b\in B} \alpha_b f(b)$ for some integers $\alpha_b,b\in B$. Restricting to the coordinates in $A$, we obtain that $w = \sum_{b\in B}\alpha_b b\in \lat(B)$.

\item[(c)] Let $w\in\lin(F)\cap \cZ^A$. Write $w = \sum_{x\in F\cap \{0,1\}^A} \lambda_x x$, and let $\tau:=\1^\top \lambda$ which is $\frac{1}{g}$-integral by the claim above. Let $w':=\sum_{x\in F\cap \{0,1\}^A} \lambda_x f(x)$, which is $\frac{1}{g}$-integral as $\tau\in \frac{1}{g}\cZ$. Subsequently, $gw'$ is an integral vector in $\lin(F')$, so $gw' = \sum_{b\in B} \alpha_b f(b)$ for some integers $\alpha_b,b\in B$, as $B'$ is an integral basis for $\lin(F')$. Restricting to the coordinates in $A$, we obtain that $gw = \sum_{b\in B}\alpha_b b\in L$, as promised.
\end{enumerate}

Observe that (b) proves part {\bf (1)}, and (c) proves part {\bf (2)} of the theorem.
\end{proof}

Let us point out a subtle detail about part (1) of \Cref{scr-theorem-extension}. A set $G$ of generators may not contain a lattice basis for $\lat(G)$. For instance, $\lat(\{2,3\})=\cZ$, yet $\{2,3\}$ does not contain a lattice basis for $\cZ$. Thus, the claim that $F\cap \{0,1\}^A$ contains a lattice basis is non-trivial.

Let's look at part (2) of \Cref{scr-theorem-extension}. This part equivalently states that, for $L:=\lat\big(F\cap \{0,1\}^A\big)$ and $\overline{L}:=\lin(F)\cap \cZ^A$, the quotient group $\overline{L}/L$ is an abelian group where the order of every element divides $g$. Subsequently, every elementary divisor of $\overline{L}/L$ divides $g$. This implies in turn that $\ind(L)$ is the product of some divisors of $g$. Furthermore, if $g$ is a prime number, then $\overline{L}/L$ is an elementary $p$-primary group. For more on concepts relating to group theory, we refer the interested reader to Dummit and Foote's excellent textbook~\cite{Dummit04}, more specifically, Chapter 5, Theorem 5.

\subsection{Proof of the main theorem}

\begin{proof}[Proof of \Cref{scr-theorem}]
Let $D=(V,A)$ be a digraph whose underlying undirected graph is $2$-edge-connected. Let $\zF$ be a nonempty family over ground set $V$ such that $\emptyset,V\notin \zF$, and the following face of $\scr(D)$ is nonempty:
	$$F:=\scr(D)\cap \left\{x\in \cR^A: x(\delta^+(U))-x(\delta^-(U))=1-|\delta^-(U)|,\, \forall U\in \zF\right\}.$$ Suppose $\gcd\{1-|\delta^-(U)|:U\in \zF\}=1$. It then follows from \Cref{scr-theorem-extension} part (1) that the lattice $L$ generated by $F\cap \{0,1\}^A$ has a lattice basis $B\subseteq F\cap \{0,1\}^A$. Furthermore, it follows from part (2) that $L=\lin(F)\cap \cZ^A$, so $\ind(L)=1$, implying in turn that $B$ is an integral basis for $\lin(F)$ contained in $F\cap \{0,1\}^A$.
\end{proof}

\subsection{Subtractive partitioning of strengthening sets}

\begin{proof}[Proof of \Cref{ARF-partition-CO}]
Let $\tau\geq 2$ be an integer, and let $D=(V,A)$ be a digraph where the minimum size of a dicut is $\tau$. We will prove that there exists an assignment $\lambda_J\in \cZ$ to every strengthening set $J$ intersecting every minimum dicut exactly once, such that $\sum_{J}\lambda_J\1_J = \1$, $\1^\top \lambda = \tau$, and $\big\{\1_J:\lambda_J\neq 0\big\}$ will be an integral basis for its linear hull.

Let $\zF$ be the family of sets $U\subset V,U\neq \emptyset$ such that $\delta^-(U)=\emptyset$ and $|\delta^+(U)|=\tau$. Let $F:=\scr(D)\cap \{x:x(\delta^+(U))-x(\delta^-(U))=1-|\delta^-(U)|,\,\forall U\in \zF\}$. Observe that $F\cap \{0,1\}^A$ corresponds to the strengthening sets of $D$ that intersect every minimum dicut exactly once.

Since every dicut of $D$ (if any) has size at least $\tau$, it follows that $|\delta^+(U)|+(\tau-1)|\delta^-(U)|\geq \tau$ for all $U\subset V,U\neq \emptyset$, implying in turn that $x^\star:=\frac{1}{\tau}\1\in F$.

Since $\zF\neq \emptyset$, then $\gcd\{1-|\delta^-(U)|:U\in \zF\}=1$, so we may apply \Cref{scr-theorem} to conclude that $F\cap \{0,1\}^A$ contains an integral basis $B$ for $\lin(F)$. This implies that $\1=\tau x^\star\in \tau F$ is an integral linear combination of the vectors in $B$, say $\sum_{b\in B} \lambda_b \cdot b$. Furthermore, $\1^\top \lambda=\tau$, because $\1^\top \lambda = \sum_{b\in B} \lambda_b \cdot b(\delta^+(U)) = \1(\delta^+(U)) = \tau$ for any given $U\in \zF$. Given that $B$ is an integral basis for $\lin(F)$, it follows that $\{b:\lambda_b\neq 0\}$ is also an integral basis for its linear hull, so we are done.
\end{proof}

\subsection{Sparse $p$-adic optimal packings of dijoins}

\begin{proof}[Proof of \Cref{p-adic-CO}]
Let $D=(V,A)$ be a digraph whose underlying undirected graph is connected. 
Denote by $M_0$ the matrix whose columns are labeled by $A$, and whose rows are the indicator vectors of the dijoins of $D$. Consider the following pair of dual linear programs:
\begin{align}
&\min\{\1^\top x:M_0x\geq \1,x\geq \0\}\label{eq:dijoin-LP}\tag{$P_0$}\\
&\max\{\1^\top y:M_0^\top y\leq \1,y\geq \0\}.\label{eq:dijoin-LP-dual}\tag{$D_0$}
\end{align}
Let $p$ be a prime number. Our goal is to exhibit a $p$-adic optimal solution to \eqref{eq:dijoin-LP-dual} with at most $2|A|$ nonzero entries. 

It is known that the basic optimal solutions of \eqref{eq:dijoin-LP} are precisely the indicator vectors of the minimum dicuts of $D$ (\cite{Lucchesi78}, see~\cite{Cornuejols01}, \S1.3.4). Let $\tau\geq 1$ be the minimum size of a dicut, which is therefore the common optimal value of the primal and dual \b{linear programs}. If $\tau=1$, then any vector $y$ that is a standard unit vector is optimal for the dual, and we are clearly done. Otherwise, $\tau\geq 2$. \b{In particular, the underlying undirected graph of $D$ is $2$-edge-connected.} Let $\zF$ be the family of sets $U\subset V,U\neq \emptyset$ such that $\delta^-(U)=\emptyset$ and $|\delta^+(U)|=\tau$. By definition, $\zF\neq \emptyset$, so $\gcd\{1-|\delta^-(U)|:U\in \zF\}=1$. 

By complementary slackness for \eqref{eq:dijoin-LP} and \eqref{eq:dijoin-LP-dual}, there exists a dijoin, and therefore a minimal dijoin~$J$, which intersects every dicut $\delta^+(U),U\in \zF$ exactly once. As every minimal dijoin of $D$ is a strengthening set, $\1_J$ belongs to the following face of $\scr(D)$: 
$$F:=\scr(D)\cap \left\{x\in \cR^A: x(\delta^+(U))-x(\delta^-(U))=1-|\delta^-(U)|,\, \forall U\in \zF\right\}.$$ As $F\neq \emptyset$, and $\gcd\{1-|\delta^-(U)|:U\in \zF\}=1$, we can apply \Cref{scr-theorem} to conclude that $F\cap \{0,1\}^A$ contains an integral basis $B$ for $\lin(F)$. 

Let $M_1$ be the matrix whose rows are the points in $\scr(D)\cap \{0,1\}^A$. Since every strengthening set is also a dijoin, it follows that $M_1$ is a row submatrix of $M_0$. Let $M_2$ and $M_4$ be the row submatrices of $M_1$ corresponding to the vectors in $F\cap \{0,1\}^A$ and $B$, respectively. \b{Our goal is to find a row submatrix $M_3$ of $M_0$, sandwiched between $M_2$ and $M_4$, with at most $|A|+|B|\leq 2|A|$ rows such that the system $M_3^\top y=\1,\1^\top y=\tau,y\geq \0$ has a $p$-adic solution. See \Cref{fig:M0-M4} for a visual illustration of the matrices.}

\b{
\begin{figure}[ht]
\begin{tikzpicture}[scale=1]    
    \def\totalheight{5}
    \def\verticalfourth{4}
    \def\verticalthird{3}
    \def\verticalsecond{2}
    \def\verticalfirst{1}
    \def\width{3}
    
    \draw[thick] (0,\verticalfourth) -- (\width,\verticalfourth);
    \node at (\width-0.3,\verticalfourth+0.3) {$M_4$};
    \node[right] at (\width+0.2, \verticalfourth+0.3) {The rows of $M_4$ are the vectors in $B$.};
    
    \draw[dashed] (0,\verticalthird) -- (\width,\verticalthird);
    \node at (\width-0.3,\verticalthird+0.3) {$M_3$};
    \node[right] at (\width+0.2, \verticalthird+0.3) {$M_3$ is a submatrix sandwiched between $M_2,M_4$ which we are after.};
    
    \draw[thick] (0,\verticalsecond) -- (\width,\verticalsecond);
    \node at (\width-0.3,\verticalsecond+0.3) {$M_2$};
    \node[right] at (\width+0.2, \verticalsecond+0.3) {The rows of $M_2$ are the points in $F\cap \{0,1\}^A$.};
    
    \draw[thick] (0,\verticalfirst) -- (\width,\verticalfirst);
    \node at (\width-0.3,\verticalfirst+0.3) {$M_1$};
    \node[right] at (\width+0.2, \verticalfirst+0.3) {The rows of $M_1$ are the points in $\scr(D)\cap \{0,1\}^A$.};
    
    \draw[thick] (0,0) rectangle (\width,\totalheight);
    \node at (\width-0.3,0.3) {$M_0$};
    \node[right] at (\width+0.2, 0.3) {The rows of $M_0$ correspond to the dijoins of $D$.};
    
\end{tikzpicture}
\caption{\b{An illustration of $M_i,0\leq i\leq 4$, where $M_i$ is a row submatrix of $M_{i-1}$ for $1\leq i\leq 4$.}}
\label{fig:M0-M4}
\end{figure}
}

\setcounter{claim}{0}

\begin{claim} 
The system $M_4^\top y=\1, \1^\top y=\tau$ has an integral solution $\lambda$.
\end{claim}
\begin{cproof}
Since every dicut of $D$ has size at least $\tau$, it follows that $|\delta^+(U)|+(\tau-1)|\delta^-(U)|\geq \tau$ for all $U\subset V,U\neq \emptyset$, implying in turn that $x^\star:=\frac{1}{\tau}\1\in F$. Subsequently, $\1=\tau x^\star$ can be expressed as an integer linear combination of the vectors in $B$, say $\1 = \sum_{b\in B} \lambda_b \cdot b$ for $\lambda_b\in \cZ,b\in B$. Note that $\frac{1}{\tau}\sum_{b\in B} \lambda_b = 1$ given that $\frac{1}{\tau}\1\in F$. Thus, $\lambda$ is the desired solution.
\end{cproof}

Consider the following pair of dual linear programs: \begin{align}
&\min\{\1^\top x:M_1x\geq \1,x\geq \0\}\label{eq:1acf-LP}\tag{$P_1$}\\
&\max\{\1^\top y:M_1^\top y\leq \1,y\geq \0\}.\label{eq:1acf-LP-dual}\tag{$D_1$}
\end{align}
Given that every strengthening set is a dijoin, and every minimal dijoin is a strengthening set, \eqref{eq:1acf-LP} is equivalent to \eqref{eq:dijoin-LP}. Thus, \eqref{eq:1acf-LP} is integral and the indicator vector of every minimum dicut of $D$ is optimal for it. Subsequently, the optimal value of \eqref{eq:1acf-LP-dual} is $\tau$. \b{Note that this dual linear program computes the maximum value of a fractional packing $y\geq \0$ of strengthening sets of $D$ such that every arc $a\in A$ has \emph{congestion} at most one, i.e., $\sum \left(y_z : z\in \scr(D)\cap \{0,1\}^A,z_a=1\right)\leq 1$.} By complementary slackness for this pair of linear programs, for any optimal solution $\bar{y}$ for \eqref{eq:1acf-LP-dual}, we have $\bar{y}_z>0$ only if $z\in F\cap \{0,1\}^A$.
\b{To see this, note that if $\bar{y}_z>0$ for some $z\in \scr(D)\cap \{0,1\}^A$, then by complementary slackness, the constraint of $M_1x\geq \1$ corresponding to $z$, i.e., $z^\top x\geq 1$, must be satisfied at equality at every optimal solution \eqref{eq:1acf-LP}. In particular, $z(\delta^+(U))=1$ for every minimum dicut $\delta^+(U)$ of $D$. Thus, $z\in F$.}

\begin{claim} 
The system $M_2^\top y=\1,\1^\top y = \tau,y\geq \0$ has a solution $y^\star$ that is strictly positive.
\end{claim}
\begin{cproof}
By \Cref{scr->dij}, there exists a digraft $(D'=(V',A'),\zF')$ such that for $F':=F(D',\zF')$, the mapping $f:F\to F'$ defined as $f(x)=\big(\begin{smallmatrix} x\\ \1-x\end{smallmatrix}\big)$ is a bijection. Let $c:=\tau f(x^\star) = \big(\begin{smallmatrix} \1\\ (\tau-1)\1\end{smallmatrix}\big)\in \cR^{A'}$. \b{The \emph{$c$-weight} of a dicut $\delta_{D'}^+(U)$ is $\sum_{a\in \delta_{D'}^+(U)} c_a$.} Then the minimum $c$-weight of a dicut of $D'$ is $\tau$, and every $\delta_{D'}^+(U),U\in \zF'$ is a {dicut of minimum $c$-weight, i.e., a \emph{minimum weight $c$-dicut}}. 

\b{
Let $M'_0$ be the matrix whose columns are labeled by $A'$, and whose rows are the indicator vectors of the dijoins of $D'$. Consider the following pair of dual linear programs:
\begin{align}
&\min\{c^\top x':M'_0x'\geq \1,x'\geq \0\}\label{eq:dijoin-LP'}\tag{$P'_0$}\\
&\max\{\1^\top y':{M'_0}^\top y'\leq c,y'\geq \0\}.\label{eq:dijoin-LP-dual'}\tag{$D'_0$}
\end{align} The joint optimal value of these linear programs is $\tau$.} 
Furthermore, {we have the following properties:}
\begin{enumerate}
\item[(i)] every arc of $A'$ appears in a minimum $c$-weight dicut of $D'$\b{: this is ensured by our construction of the digraph $D'$},
\item[(ii)] if $J'$ is a dijoin of $D'$ which intersects every minimum $c$-weight dicut of $D'$ exactly once, then $\1_{J'} \in F'\cap \{0,1\}^{A'}${: this follows from the definition of $F'$, and the fact that every $\delta^+_{D'}(U),U\in \zF'$ is a minimum $c$-weight dicut,}
\item[(iii)] if $\1_{J'} \in F'\cap \{0,1\}^{A'}$, then $J'$ is a dijoin of $D'$ which intersects every minimum $c$-weight dicut of $D'$ exactly once. \b{To this end, pick a minimum $c$-weight dicut $C:=\delta_{D'}(U')\subseteq A'$ of $D'$, which has $c$-weight $\tau$. If $|U'|=|V'|-1$, then clearly $|C\cap J'|=1$. 

Otherwise, $|U'|<|V'|-1$. Let $U:=U'\cap V$. Then $C = \delta^+_{D'}(\varphi(U))$ and $|C|=|\delta_D(U)|$. As the underlying undirected graph of $D$ is $2$-edge-connected, it follows that $|C|\geq 2$. Choose $\1_J\in F\cap \{0,1\}^A$ such that $f(\1_J) = \1_{J'}$. Then $|C\cap J'| = |\delta^+_D(U)\cap J|+|\delta^-_D(U)|-|\delta^-_D(U)\cap J|$. (See the proof of \Cref{scr->dij} for more explanation.)

If $|C|\geq 3$, then $c_a=1$ for all $a\in C$, so $\delta_D^+(U)=C\subseteq A$ is a dicut of $D$ of minimum size. Subsequently, $U\in \zF$, so $|J\cap C|=1$ as $\1_J\in F$. Given that $|C\cap J'| = |C\cap J|$ in this case, it follows that $|C\cap J'|=1$.

If $|C|=2$, then $|\delta_D(U)|=2$, so in any strongly connected re-orientation of $D$, exactly one arc will leave $U$, implying in turn that $|C\cap J'| = |\delta^+_D(U)\cap J|+|\delta^-_D(U)|-|\delta^-_D(U)\cap J|=1$, as required.
}
\end{enumerate}

As (iii) holds, it follows from strict complementarity that there exists {an optimal solution} $\hat{y}$ {of \eqref{eq:dijoin-LP-dual'}} such that $\hat{y}_{J'}>0$ for every $J'$ such that $\1_{J'} \in F'\cap \{0,1\}^{A'}$. 
As (ii) holds, it follows from complementary slackness that $\hat{y}_{J'}=0$ for every dijoin $J'$ such that $\1_{J'} \notin F'\cap \{0,1\}^{A'}$. 
Finally, as (i) holds, it follows from complementary slackness that \b{for} every arc \b{$a$} of $A'$, {we have that $\sum\left(\hat{y}_{J'} :  \1_{J'}\in F'\cap \{0,1\}^{A'}\right) = c_a$}. 

Let $y^\star\in \cR_{\geq 0}^{F\cap \{0,1\}^A}$ be defined as follows: 
for every $z\in F\cap \{0,1\}^A$, let $y^\star_z = \hat{y}_{J'}$ where $\1_{J'} = f(z)$. 
Since every arc {$a$ of $A$ has congestion exactly $c_a=1$ in $\hat{y}$}, and since $\hat{y}_{J'}=0$ for every dijoin $J'$ such that $\1_{J'} \notin F'\cap \{0,1\}^{A'}$, it follows that $M_2^\top y^\star = \1$. By construction, $y^\star_z>0$ for all $z\in F\cap \{0,1\}^A$. Finally, $\1^\top y^\star = \1^\top \hat{y} = \tau$, so $y^\star$ is the desired optimal solution for \eqref{eq:1acf-LP-dual}.
\end{cproof}

Consider now the polyhedron $P=\{y:M_2^\top y=\1,\1^\top y = \tau, y\geq \0\}$. The existence of $y^\star$ from Claim~2 implies that $\aff(P) = \{y:M_2^\top y = \1, \1^\top y=\tau\}$. Furthermore, the existence of $\lambda$ from Claim~1 implies that $\aff(P)$ contains an integral, hence $p$-adic point (recall that $M_4$ is a row submatrix of $M_2$). Thus, it follows from (\cite{Abdi24-TDD}, Lemma 2.2) that $P$ contains a $p$-adic point. Our goal is to find a $p$-adic point in $P$ of \emph{support size} at most $|A|+|B|$.

We know from Claim~1 that the system $M_4^\top y =\1,\1^\top y = \tau$ in $|B|$ variables has an integral, hence $p$-adic solution. This system, however, may not have a nonnegative $p$-adic solution. In what follows, we argue that after adding at most $|A|$ rows from $M_2$ to $M_4$ we can guarantee a nonnegative solution as well.

Consider an optimal basic feasible solution $(y^\star,x_0^\star,x^\star)$ to the following bounded linear program \begin{equation*}
\max\left\{x_0:M_2^\top y=\1; \1^\top y = \tau; y_{b}- x_0-x_b=0,\,\forall b\in B ; y\geq \0;x_b\geq 0,\,\forall b\in B; x_0\geq 0\right\}
\end{equation*} which is in standard equality form. Since $M_2^\top y=\1,\1^\top y=\tau, y>\0$ is feasible by Claim~2, it follows that $x_0^\star>0$, so $y^\star_{b}>0$ for all $b\in B$. As a basic feasible solution, $(y^\star,x_0^\star,x^\star)$ has support size bounded above by the rank of the coefficient matrix for the equality constraints, which is at most $|A|+1+|B|$. Denote by $M_3$ the row submatrix of $M_2$ corresponding to the nonzero entries of $y^\star$, of which there are at most $|A|+|B|$. Clearly, $M_3^\top y=\1,\1^\top y=\tau$ has a solution where every entry is greater than $0$. Moreover, by design, $M_3$ contains $M_4$ and at most $|A|$ other rows from $M_2$. In particular, $M_3^\top y=\1,\1^\top y=\tau$ has an integral solution, too.

In summary, we found a row submatrix $M_3$ sandwiched between $M_2$ and $M_4$, with at most $|A|+|B|$ rows where $M_3^\top y=\1,\1^\top y=\tau,y\geq \0$ is feasible, and the affine hull $\{y:M_3^\top y=\1, \1^\top y=\tau\}$ contains an integral, hence $p$-adic point. Therefore, by (\cite{Abdi24-TDD}, Lemma 2.2), the system $M_3^\top y=\1,\1^\top y=\tau,y\geq \0$ has a $p$-adic solution, implying in turn that $M_0^\top y = \1, \1^\top y=\tau, y\geq \0$ has a $p$-adic solution with support size at most $|A|+|B|\leq 2|A|$, as required.
\end{proof}

\subsection{Strongly connected orientations of hypergraphs}

\begin{proof}[Proof of \Cref{head-disjoint-CO}]
Let $\tau\geq 2$ be an integer, and let $H=(V,\zE)$ be a $\tau$-uniform hypergraph such that $d_H(X)\geq \tau$ for all $X\subset V,X\neq \emptyset$. Our goal is to find an assignment $\lambda_O\in \cZ$ to every strongly connected orientation $O:\zE\to V$ such that $$
\sum_{O(E) = v} \lambda_O =1 \qquad \forall E\in \zE,\, \forall v\in E,
$$ and $|\{O:\lambda_O\neq 0\}|\leq (\tau-1)|\zE|+1$.

To this end, let $D$ be the bipartite digraph on node set $V\cup \{t_E:E\in \zE\}$ and arc set $\{(v,t_E):v\in E,E\in \zE\}$. In words, we have introduced a sink $t_E$ for every hyperedge $E\in \zE$, and added an arc from every node in $E$ to $t_E$. Thus, every node in $V$ is a source, and every node in $\{t_E:E\in \zE\}$ is a sink of degree $\tau$, as $H$ is $\tau$-uniform.

We have that $d_H(X)\geq \tau$ for all $X\subset V,X\neq \emptyset$, which states equivalently that every dicut of $D$ has size at least $\tau$. In particular, as $\tau\geq 2$, the underlying undirected graph of $D$ is $2$-edge-connected. 

Since the minimum size of a dicut of $D$ is $\tau$, there exists a fractional packing $y^\star$ of dijoins of $D$ of value $\tau$. Since every sink has degree $\tau$, it follows from complementary slackness that if $y^\star_J>0$, then $|J\cap \delta_D(t_E)|=1$ for all $E\in \zE$. Since every arc belongs to a minimum dicut, it follows that $\sum_{J} y^\star_J\1_J=\1$.

Let $\zF:=\{V(D)\setminus t_E:E\in \zE\}$. What we argued above implies that whenever $y^\star_J>0$, then $\1_J\in F:=F(D,\zF)$. In particular, $F\neq \emptyset$, so $(D,\zF)$ is a digraft. Thus, by \Cref{main-digraft}, $F\cap \{0,1\}^{A(D)}$ contains an integral basis $B$ for $\lin(F)$. As $\1=\sum_{J} y^\star_J\1_J\in \lin(F)$, it follows that $\1=\sum_{b\in B}\alpha_b b$ for some integers $\alpha_b,b\in B$.

For every orientation $O:\zE\to V$ of $H$, let $J_O:=\{(O(E),t_E):E\in \zE\}\subseteq A(D)$. It can be readily checked that if $J_O$ is a dijoin of $D$, i.e., if $\1_{J_O}\in F$, then $O$ is a strongly connected orientation of $H$. 

Observe that if $\1_J\in F$, then $J=J_O$ for some orientation $O$ of $H$, which must be strongly connected as argued above. Thus, every point in $F\cap \{0,1\}^{A(D)}$ corresponds to some strongly connected orientation of $H$. For every $b\in B$, let $O_b$ be the corresponding strongly connected orientation in $H$, and let $\lambda_{O_b}:=\alpha_b$; let $\lambda_O:=0$ for all other strongly connected orientations $O$ of $H$. The equality $\sum_{b\in B}\alpha_b b_{(v,t_E)}=1$ for all $(v,t_E)\in A(D)$, implies that $$
\sum_{O(E) = v} \lambda_O =1 \qquad \forall E\in \zE,\, \forall v\in E.
$$ Furthermore, $|\{O:\lambda_O\neq 0\}|\leq |B|\leq |A(D)|-|\zE|+1 = (\tau-1)|\zE|+1$, as desired.
\end{proof}

\subsection{An example and a conjecture}\label{subsec:schrijver-example}

In this subsection, we give a classic example illustrating that three of our theorems are best possible in a certain sense, and also state a conjecture about an extension of \Cref{scr-theorem} to faces of the strongly connected re-orientations polytope where some capacity constraints are also fixed to equality.

Consider the digraph $D=(V,A)$ shown in \Cref{fig:schrijver}, and let $C$ be the set of solid arcs~\cite{Schrijver80}. Denote by $W$ the set of the six sources and sinks of $D$. It can be seen that there are exactly four strengthening sets $J\subseteq C$ of $D$ such that $|J\cap \delta(v)|=1$ for all $v\in W$. That is, the face $F$ of $\scr(D)$ obtained by enforcing the following equalities has exactly four integral vectors:
\begin{align}
	x(\delta^+(v))-x(\delta^-(v)) &= 1-|\delta^-(v)| &\forall v\in W, v \text{ is a source} \label{schrijver-1}\\
	x(\delta^+(V\setminus v))-x(\delta^-(V\setminus v)) &=1-|\delta^-(V\setminus v)| &\forall v\in W, v \text{ is a sink}\label{schrijver-2}\\
	x_a&=0 &\forall a\in A\setminus C. \label{schrijver-3}
\end{align} Observe that the greatest common divisor of the right-hand side values is $1$. Denote by $L\subseteq \cZ^A$ the lattice generated by the four integral vectors in $F$. 

It can be readily checked that while $\1_C\in \cZ^A$ is an integral vector in the linear, in fact conic hull of $L$, it does not belong to $L$ itself, implying that $\ind(L)>1$, and also $F\cap \{0,1\}^A$ is not an IGSC. Thus, \Cref{scr-theorem} does not extend to all faces of $\scr(D)$ where the greatest common divisor of the right-hand sides of the tight constraints is $1$. That said, we suspect the following is true in general.

\begin{CN}\label{scr-main-CN}
	Let $D=(V,A)$ be a digraph whose underlying undirected graph is $2$-edge-connected. Let $F$ be a face of $\scr(D)$ where the greatest common divisor of the right-hand side values of the tight constraints is $1$. 
	Then the following statements hold: \begin{enumerate}
\item The lattice generated by $F\cap \{0,1\}^A$ has a lattice basis contained in $F\cap \{0,1\}^A$.
\item $2x\in \lat\left(F\cap \{0,1\}^A\right)$ for all $x\in \lin(F)\cap \cZ^A$.
\end{enumerate} 
	\end{CN}

Moving on, let $F'$ be the face of $\scr(D)$ obtained by enforcing just \eqref{schrijver-1} and \eqref{schrijver-2}. We know from \Cref{scr-theorem} that $F'\cap \{0,1\}^A$ contains an integral basis. However, as $F\cap \{0,1\}^A$ is not an IGSC, it follows that $F'\cap \{0,1\}^A$ is not an IGSC either; this is because $F$ is a face of $F'$ and so $\cone(F)$ is a face of $\cone(F')$, and being an IGSC is closed under taking faces of the conic hull. Thus, \Cref{scr-theorem} cannot be strengthened to conclude that $F\cap \{0,1\}^A$ is an IGSC.

\begin{figure}[ht]
	\centering
	\includegraphics[scale=0.3]{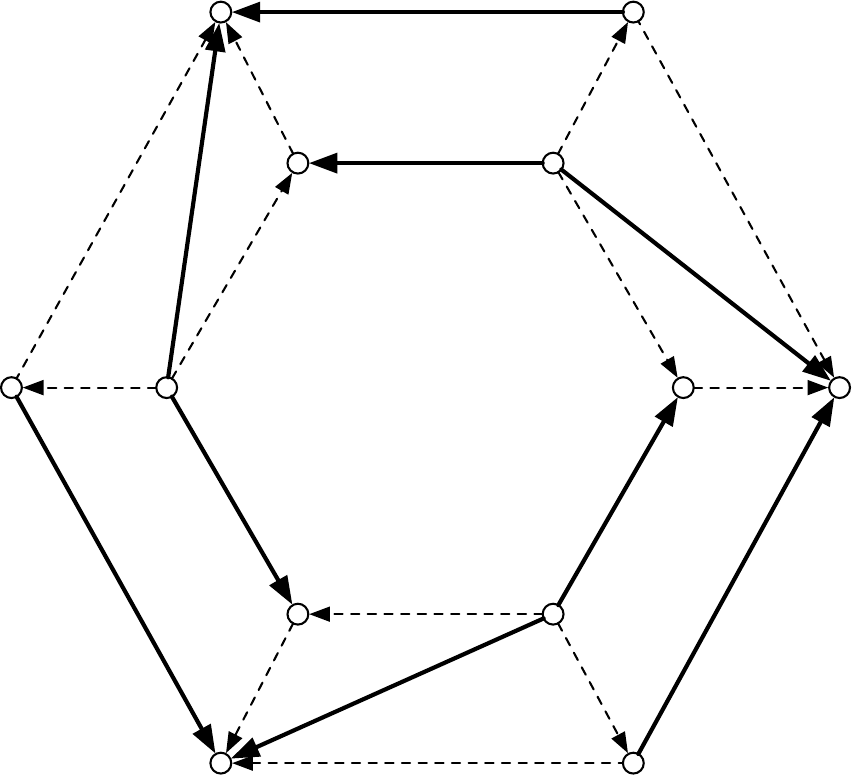}
	\caption{A $0,1$-weighted digraph where the solid and dashed arcs have weights $1$ and $0$, respectively. 	
	}
	\label{fig:schrijver}
\end{figure}

Let us view the arcs inside and outside of $C$ as having capacity $1$ and $0$, respectively. It can be readily checked that every dicut $\delta^+(U)$ has capacity at least $2$, i.e., $|C\cap \delta^+(U)|\geq 2$. Furthermore, every node in $W$ corresponds to a dicut of capacity $2$, and these are the only minimum capacity dicuts. Subsequently, the strengthening sets of $D$ contained in $C$ that intersect every minimum dicut exactly once, correspond precisely to the four $0,1$ vectors in $\scr(D)$ satisfying \eqref{schrijver-1}-\eqref{schrijver-3}. However, as we noted before, $\1_C$ cannot be expressed as an integer linear combination of these four $0,1$ vectors, showing that \Cref{ARF-partition-CO} does not extend to the capacitated setting. 

Finally, \Cref{p-adic-CO} is best possible in the sense that it does not extend to the capacitated setting; let us elaborate. While every dicut of the instance above has capacity at least $2$, the set $C$ cannot be decomposed into $2$ dijoins~\cite{Schrijver80}. What's more striking about this example is that for any prime number $p\neq 2$, there is no assignment of a $p$-adic rational number $\lambda_J$ to every dijoin \emph{contained in $C$} such that $\1^\top \lambda=2$, $\sum_{J} \lambda_J \1_J \leq \1$, and $\lambda\geq 0$. We leave this as an exercise for the reader.

\section*{Acknowledgements}

We would like to thank Krist\'{o}f B\'{e}rczi, Karthekeyan Chandrasekaran, Bertrand Guenin, and Levent Tun\c{c}el for fruitful discussions about various aspects of this work. \b{We would also like to thank anonymous reviewers whose comments vastly improved the exposition of our paper.} This work was supported in part by EPSRC grant EP/X030989/1 and ONR grant N00014-22-1-2528.

\paragraph{Data Availability Statement.} No data are associated with this article. Data sharing is not applicable to this article.

{\small \bibliographystyle{abbrv}\bibliography{references}}

\end{document}